\documentclass[a4paper,twoside]{article}
\usepackage{a4}
\usepackage{amssymb}
\usepackage{amsmath}
\usepackage{upref}
\usepackage[active]{srcltx}
\usepackage[pagebackref,colorlinks,citecolor=blue,linkcolor=blue]{hyperref}
\usepackage[dvipsnames]{color}
\allowdisplaybreaks[2] % To make it possible for displaybreaks
\newcount\minutes \newcount\hours
\hours=\time
\divide\hours 60
\minutes=\hours
\multiply\minutes -60
\advance\minutes \time
\newcommand{\klockan}{\the\hours:{\ifnum\minutes<10 0\fi}\the\minutes}
\newcommand{\tid}{\today\ \klockan}
\newcommand{\prtid}{\smash{\raise 10mm \hbox{\LaTeX ed \tid}}}
\renewcommand{\prtid}{}
\makeatletter
\def\sectionmark#1{} %\markboth{{\sectnr #1}}{{\sectnr #1}}} %Journal
\def\subsectionmark#1{}
\newcommand{\sectnr}{\ifnum \c@secnumdepth >\z@
                 \thesection.\hskip 1em\relax \fi}
\def\@evenhead{\footnotesize\rm\thepage\hfil\leftmark\hfil\llap{\prtid}}
\def\@oddhead{\footnotesize\rm\rlap{\prtid}\hfil\rightmark\hfil\thepage}
\def\tableofcontents{\section*{Contents} %\@mkboth{Contents}{Contents}} %Journal
 \@starttoc{toc}}
\makeatother
\makeatletter
\def\@biblabel#1{#1.}
\makeatother
\makeatletter
\let\Thebibliography=\thebibliography
\renewcommand{\thebibliography}[1]{\def\@mkboth##1##2{}\Thebibliography{#1}
\addcontentsline{toc}{section}{References}
\frenchspacing % Maybe not needed
\setlength{\@topsep}{0pt}% Delete if extra space before list
\setlength{\itemsep}{0pt}%
\setlength{\parskip}{0pt plus 2pt}%
}
\makeatother
\makeatletter
\def\mdots@{\mathinner.\nonscript\!.%
 \ifx\next,.\else\ifx\next;.\else\ifx\next..\else
 \nonscript\!\mathinner.\fi\fi\fi}
\let\ldots\mdots@
\let\cdots\mdots@
\let\dotso\mdots@
\let\dotsb\mdots@
\let\dotsm\mdots@
\let\dotsc\mdots@
\def\vdots{\vbox{\baselineskip2.8\p@ \lineskiplimit\z@
    \kern6\p@\hbox{.}\hbox{.}\hbox{.}\kern3\p@}}
\def\ddots{\mathinner{\mkern1mu\raise8.6\p@\vbox{\kern7\p@\hbox{.}}%
    \raise5.8\p@\hbox{.}\raise3\p@\hbox{.}\mkern1mu}}
\makeatother
\makeatletter
\let\Enumerate=\enumerate
\renewcommand{\enumerate}{\Enumerate%
\setlength{\@topsep}{0pt}% Delete if extra space before list
\setlength{\itemsep}{0pt}%
\setlength{\parskip}{0pt plus 1pt}%
\renewcommand{\theenumi}{\textup{(\alph{enumi})}}%
\renewcommand{\labelenumi}{\theenumi}%
}
\let\endEnumerate=\endenumerate
\renewcommand{\endenumerate}{\endEnumerate\unskip}
\makeatother
\makeatletter
\def\@seccntformat#1{\csname the#1\endcsname.\quad}
\makeatother
\newcommand{\authortitle}[2]{\author{#1}\title{#2}\markboth{#1}{#2}}
\newcommand{\auth}[2]{{#1, #2.}}
\newcommand{\art}[6]{{\sc #1, \rm #2, \it #3\/ \bf #4 \rm (#5), \mbox{#6}.}}
\newcommand{\artprep}[3]{{\sc #1, \rm #2, \rm #3.}}
\newcommand{\artin}[3]{{\sc #1, \rm #2,  in #3.}}
\newcommand{\arttoappear}[3]{{\sc #1, \rm #2, to appear in \it #3}}
\newcommand{\book}[3]{{\sc #1, \it #2, \rm #3.}}
\newcommand{\AND}{{\rm and }}
\RequirePackage{amsthm}
\newtheoremstyle{descriptive}%
  {\topsep}   %{\medskipamount}          % Space above
  {\topsep}   %  {\medskipamount}          % Space below
  {\rmfamily} % Body font
  {}          % Indent
  {\bfseries} % Head font
  {.}         % Punctuation after thm head
  { }         % Space after thm head
  {}          % Thm head spec(?)
\newtheoremstyle{propositional}%
  {\topsep}   %  {\medskipamount}          % Space above
  {\topsep}   %  {\medskipamount}          % Space below
  {\itshape}  % Body font
  {}          % Indent
  {\bfseries} % Head font
  {.}         % Punctuation after thm head
  { }         % Space after thm head
  {}          % Thm head spec(?)
\theoremstyle{propositional}
\newtheorem{thm}{Theorem}[section]
\newtheorem{prop}[thm]{Proposition}
\newtheorem{lem}[thm]{Lemma}
\theoremstyle{descriptive}
\newtheorem{deff}[thm]{Definition}
\newtheorem{example}[thm]{Example}
\newtheorem{remark}[thm]{Remark}
\makeatletter
\renewenvironment{proof}[1][\proofname]{\par
  \pushQED{\qed}%
  \normalfont
  \trivlist
  \item[\hskip\labelsep
        \itshape
    #1\@addpunct{.}]\ignorespaces
}{%
  \popQED\endtrivlist\@endpefalse
}
\makeatother
\newcommand{\setm}{\setminus}
\renewcommand{\emptyset}{\varnothing}
\def\vint{\mathop{\mathchoice%
          {\setbox0\hbox{$\displaystyle\intop$}\kern 0.22\wd0%
           \vcenter{\hrule width 0.6\wd0}\kern -0.82\wd0}%
          {\setbox0\hbox{$\textstyle\intop$}\kern 0.2\wd0%
           \vcenter{\hrule width 0.6\wd0}\kern -0.8\wd0}%
          {\setbox0\hbox{$\scriptstyle\intop$}\kern 0.2\wd0%
           \vcenter{\hrule width 0.6\wd0}\kern -0.8\wd0}%
          {\setbox0\hbox{$\scriptscriptstyle\intop$}\kern 0.2\wd0%
           \vcenter{\hrule width 0.6\wd0}\kern -0.8\wd0}}%
          \mathopen{}\int}
\newcommand{\Cp}{{C_p}}
\newcommand{\CpOm}{{C_p^\Om}}
\newcommand{\din}{d_{\rm in}}
\newcommand{\Bin}{B_{\rm in}}

\newcommand{\Bij}{B^{i,j}}
\newcommand{\Bhij}{\widehat{B}^{i,j}}
\newcommand{\Bhi}{\widehat{B}^{i}}
\newcommand{\Bhx}{\widehat{B}_{x'}}
\newcommand{\Bhxk}{\widehat{B}_{x'_k}}

\newcommand{\xij}{x_{i,j}}
\newcommand{\tij}{t_{i,j}}
\DeclareMathOperator{\sgn}{sgn}
\DeclareMathOperator{\Lip}{Lip}
\newcommand{\Lipc}{{\Lip_c}}
\DeclareMathOperator{\dist}{dist}
\DeclareMathOperator{\diam}{diam}
\DeclareMathOperator{\length}{length}
\DeclareMathOperator{\supp}{supp}
\DeclareMathOperator{\interior}{int}
\DeclareMathOperator*{\essliminf}{ess\,lim\,inf}
\newcommand{\loc}{_{\rm loc}}
{\catcode`p =12 \catcode`t =12 \gdef\eeaa#1pt{#1}}      % Get slantfactor
\def\accentadjtext#1{\setbox0\hbox{$#1$}\kern   % Convert it with height
                \expandafter\eeaa\the\fontdimen1\textfont1 \ht0 }
\def\accentadjscript#1{\setbox0\hbox{$#1$}\kern % Convert it with height
                \expandafter\eeaa\the\fontdimen1\scriptfont1 \ht0 }
\def\accentadjscriptscript#1{\setbox0\hbox{$#1$}\kern   % Convert it with height
                \expandafter\eeaa\the\fontdimen1\scriptscriptfont1 \ht0 }
\def\accentadjtextback#1{\setbox0\hbox{$#1$}\kern       % Convert it with height
                -\expandafter\eeaa\the\fontdimen1\textfont1 \ht0 }
\def\accentadjscriptback#1{\setbox0\hbox{$#1$}\kern     % Convert it with height
                -\expandafter\eeaa\the\fontdimen1\scriptfont1 \ht0 }
\def\accentadjscriptscriptback#1{\setbox0\hbox{$#1$}\kern % Convert it with height
                -\expandafter\eeaa\the\fontdimen1\scriptscriptfont1 \ht0 }
\def\itoverline#1{{\mathsurround0pt\mathchoice
        {\rlap{$\accentadjtext{\displaystyle #1}
                \accentadjtext{\vrule height1.593pt}
                \overline{\phantom{\displaystyle #1}
                \accentadjtextback{\displaystyle #1}}$}{#1}}
        {\rlap{$\accentadjtext{\textstyle #1}
                \accentadjtext{\vrule height1.593pt}
                \overline{\phantom{\textstyle #1}
                \accentadjtextback{\textstyle #1}}$}{#1}}
        {\rlap{$\accentadjscript{\scriptstyle #1}
                \accentadjscript{\vrule height1.593pt}
                \overline{\phantom{\scriptstyle #1}
                \accentadjscriptback{\scriptstyle #1}}$}{#1}}
        {\rlap{$\accentadjscriptscript{\scriptscriptstyle #1}
                \accentadjscriptscript{\vrule height1.593pt}
                \overline{\phantom{\scriptscriptstyle #1}
                \accentadjscriptscriptback{\scriptscriptstyle #1}}$}{#1}}}}
\newcommand{\al}{\alpha}
\newcommand{\alp}{\alpha}
\newcommand{\be}{\beta}
\newcommand{\ga}{\gamma}
\newcommand{\de}{\delta}
\newcommand{\eps}{\varepsilon}
\newcommand{\la}{\lambda}
\newcommand{\La}{\Lambda}
\newcommand{\Om}{\Omega}
\renewcommand{\phi}{\varphi}
\newcommand{\p}{{$p\mspace{1mu}$}}
\newcommand{\R}{\mathbf{R}}
\newcommand{\eR}{{\overline{\R}}}
\newcommand{\Np}{N^{1,p}}
\newcommand{\Nq}{N^{1,q}}
\newcommand{\Nploc}{N^{1,p}\loc}
\newcommand{\Nqloc}{N^{1,q}\loc}
\newcommand{\hNp}{\widehat{N}^{1,p}}
\newcommand{\Lploc}{L^p\loc}
\newcommand{\Lp}{L^p}
\newcommand{\rp}{r'}
\newcommand{\Ga}{\Gamma}
\newcommand{\shat}{{\hat{s}}}
\newcommand{\xh}{{\hat{x}}}
\newcommand{\qhat}{{\hat{q}}}
\newcommand{\ut}{{\tilde{u}}}
\newcommand{\ft}{\tilde{f}}
\newcommand{\clBzeroprime}{{\itoverline{B_0'}}}
\newcommand{\CPI}{C_{\rm PI}}
\newcommand{\ub}{\bar{u}}

\newcommand{\Etau}{E_\tau}
\numberwithin{equation}{section}
\newcommand{\imp}{\mathchoice{\quad \Longrightarrow \quad}{\Rightarrow}
                {\Rightarrow}{\Rightarrow}}
\newenvironment{ack}{\medskip{\it Acknowledgement.}}{}

\begin{document}

\authortitle{Anders Bj\"orn and Jana Bj\"orn}
{Local and semilocal Poincar\'e inequalities
on metric spaces}

\author{
Anders Bj\"orn \\
\it\small Department of Mathematics, Link\"oping University, \\
\it\small SE-581 83 Link\"oping, Sweden\/{\rm ;}
\it \small anders.bjorn@liu.se
\\
\\
Jana Bj\"orn \\
\it\small Department of Mathematics, Link\"oping University, \\
\it\small SE-581 83 Link\"oping, Sweden\/{\rm ;}
\it \small jana.bjorn@liu.se
}

\date{To appear in \emph{J. Math. Pures Appl.}}

\maketitle

\noindent{\small {\bf Abstract}. 
We consider several local versions of the doubling condition and Poincar\'e
inequalities on metric measure spaces. 
Our first result is that in proper connected spaces, the weakest local
assumptions self-improve to semilocal ones, i.e.\ holding within every ball.

We then study various geometrical and analytical consequences of such
local assumptions, such as local quasiconvexity, self-improvement of
Poincar\'e inequalities, existence of Lebesgue points, density of 
Lipschitz functions and quasicontinuity of Sobolev functions.
It turns out that local versions of these properties hold under local
assumptions, even though they are not always straightforward.

We also conclude that many qualitative, as well as quantitative, properties
of \p-harmonic functions on metric spaces can be proved in various forms
under such local assumptions, with the main exception being the Liouville
theorem, which fails without global assumptions.

\medskip

\noindent{\bf R\'esum\'e.}
Nous consid\'erons plusieurs versions locales des conditions de
doublement et des in\'egalit\'es de Poincar\'e dans des espaces
m\'etriques mesur\'es. Notre premier r\'esultat stipule que dans un espace
propre connexe, les hypoth\`eses locales les plus faibles
s'am\'eliorent en semi-locales, c.\`a.d. elles sont valables dans
chaque boule. 
	
Nous \'etudions ensuite certaines cons\'equences g\'eom\'etriques et
analytiques de telles hypoth\`eses locales tel que la
quasi-convexit\'e locale, l'auto am\'elioration des in\'egalit\'es de
Poincar\'e, l'existence des points Lebesgue, la densit\'e des
fonctions Lipschitz et la quasi-continuit\'e des fonctions Sobolev. Il
s'av\`ere que les versions locales de ces propri\'et\'es restent
valables sous les hypoth\`eses locales m\^eme si elles ne sont pas
toujours imm\'ediates.  
	
Nous concluons \'egalement que sous telles hypoth\'eses locales,
plusieurs propri\'et\'es qualitatives, ainsi que quantitatives, des
fonctions p-harmoniques sur des espaces m\'etriques peuvent \^etre
prouv\'ees sous diverses formes, l'exception principale \'etant le
th\'eor\`eme de Liouville qui \'echoue sans hypoth\'eses globales. 

\bigskip
\noindent
{\small \emph{Key words and phrases}:
capacity,
density of Lipschitz functions,
Lebesgue point,
local doubling,
metric measure space, 
Newtonian space,
nonlinear potential theory, 
\p-harmonic function,
Poincar\'e inequality,
quasicontinuity,
quasiminimizer,
semilocal doubling,
Sobolev space.
}

}
\section{Introduction}

In the last two decades,
an extensive part of first-order analysis, 
such as the
Sobolev space theory and 
nonlinear potential theory for \p-harmonic functions, 
has been developed on metric spaces.
Standard assumptions are very often that $\mu$ is a 
doubling measure
supporting a 
\p-Poincar\'e inequality and that the space $X$ is complete.
The doubling condition controls changes in scales, while the Poincar\'e inequality
guarantees that functions are controlled by their so-called upper gradients.
Both of these conditions play a vital role in many proofs.
These assumptions are usually imposed globally on the whole space.

In this paper we study how these conditions
can be relaxed and replaced by
similar local or semilocal assumptions, while retaining most of the important 
consequences.
We assume throughout the paper that $1 \le  p<\infty$
and that $X=(X,d,\mu)$ is a metric space equipped
with a metric $d$ and a positive complete  Borel  measure $\mu$
such that $0<\mu(B)<\infty$ for all  balls $B \subset X$.

\begin{deff} 
The measure $\mu$ is \emph{locally doubling}  if 
for every $x \in X$ there are $r,C>0$ 
(depending on $x$) such that 
$\mu(2B)\le C \mu(B)$ 
for all balls $B \subset B(x,r)$.
If such a $C$ exists for all $x \in X$ and all $r>0$, then 
$\mu$ is \emph{semilocally doubling}.
(Here and below, $\la B$ stands for the ball concentric with $B$
and with $\la$-times the radius.)

The measure $\mu$ supports a \emph{local \p-Poincar\'e inequality}, 
$p \ge1$,
if
for every $x \in X$ there are $r,C>0$ and $\lambda \ge 1$
such that for all balls $B\subset B(x,r)$, 
all integrable functions $u$ on $\la B$
and all upper gradients $g$ of $u$, 
\begin{equation}   \label{eq-def-local-PI-intro}
        \vint_{B} |u-u_B| \,d\mu
        \le C r_B \biggl( \vint_{\lambda B} g^{p} \,d\mu \biggr)^{1/p},
\end{equation}
where $ u_B   :=\vint_B u \,d\mu := \int_B u\, d\mu/\mu(B)$.
If such $C$ and $\la$ exist for all $x$ and all $r$, then 
$X$ supports a \emph{semilocal \p-Poincar\'e inequality}.
\end{deff}

It may be worth comparing this  with the definition of local 
function spaces. 
A function $f \in L^1\loc(X)$ if for every $x \in X$ there is $r>0$
such $f \in L^1(B(x,r))$.
If $X$ is proper  (i.e.\ all closed bounded sets are compact), 
then it is a well-known (and 
useful) fact that this is equivalent to requiring that
$f \in L^1\loc(B)$ for every ball $B$ in $X$.
It turns out that a similar fact is true for the doubling property.
Note that if $\mu$ is globally doubling, then $X$ is   proper
if and only if it is complete.

\begin{prop} \label{prop-semilocal-doubling-intro}
If $X$ is proper and $\mu$ is locally doubling, 
then $\mu$ is semilocally doubling.
\end{prop}

This is perhaps not very surprising, and essentially just needs a 
standard compactness
argument to be shown. That a similar result is true also for Poincar\'e inequalities
is far less obvious, and requires several pages to prove. 
Poincar\'e inequalities are intimately related to
connectivity properties,
and thus the precise statement is as follows.
The assumptions of properness and connectedness cannot be dropped.

\begin{thm} \label{thm-PI-intro}
If $X$ is proper and connected and $\mu$ is locally doubling
and supports a local \p-Poincar\'e inequality, then 
it supports a semilocal \p-Poincar\'e inequality.
\end{thm}

As already mentioned, in much
of the metric space literature  on first-order analysis
it is assumed  that
$\mu$ is globally doubling and supports a global \p-Poincar\'e inequality,
while it is folklore that much of the theory holds under local assumptions.
For instance, the assumptions are global in the monographs
Haj\l asz--Koskela~\cite{HaKo}, Bj\"orn--Bj\"orn~\cite{BBbook} and
Heinonen--Koskela--Shanmugalingam--Tyson~\cite{HKSTbook}.
In \cite{HaKo} and \cite{BBbook} it is mentioned
that Riemannian manifolds with Ricci curvature bounded from below
support
semilocal assumptions  
with the implicit implication that most of the
developed theory holds under these weaker assumptions.

There are some papers requiring local assumptions, but they often take
different forms from paper to paper.
Sometimes
it is assumed that the constants involved are uniform, something
we do not assume (but for 
Section~\ref{sect-local-unif}). 
Such 
assumptions (of different kinds)
are e.g.\ assumed in Cheeger~\cite{Cheeg},
Danielli--Garofalo--Marola~\cite{DaGaMa}, Garofalo--Marola~\cite{GaMa} 
and Holopainen--Shanmugalingam~\cite{HoSh}. 
The requirements are in all cases more restrictive
than our local assumptions,
and those in \cite{Cheeg} are more restrictive
than our semilocal assumptions. 

Once we have established 
Proposition~\ref{prop-semilocal-doubling-intro}
and Theorem~\ref{thm-PI-intro} (which 
we do in Sections~\ref{sect-doubling} and~\ref{sect-PI}),
we take a look at which useful
consequences of global doubling and Poincar\'e inequalities
can be obtained already under (semi)local assumptions,
with and without properness.
We study the self-improvement of
Poincar\'e inequalities, Lebesgue points, density of 
Lipschitz and locally Lipschitz functions, quasicontinuity and 
\p-harmonic functions under (semi)local assumptions.

In Section~\ref{sect-better-PI} we concentrate on 
the self-improvement of Poincar\'e inequalities and prove two 
results: one improving the norm on the left-hand side
of  \eqref{eq-def-local-PI-intro} and the other one
the norm on the right-hand
side.
In Section~\ref{sect-local-unif} we see how 
uniformly local assumptions
give slightly stronger self-improvement conclusions.
Under global assumptions these important results
are due to  Haj\l asz--Koskela~\cite{HaKo-CR}, 
\cite{HaKo} and Keith--Zhong~\cite{KZ}, respectively.
In particular, the latter result can be localized in the following
way.

\begin{thm} \label{thm-PI-uniform-q-intro}
If $X$ is locally compact and $\mu$ is locally 
doubling and supports a local \p-Poincar\'e inequality,
both with uniform constants $C$ and $\la$ and with $p>1$,
then $X$ supports a local $q$-Poincar\'e inequality
for some $q<p$ with new uniform constants $C$ and $\la$.
\end{thm}

Neither in the assumptions nor in the conclusion
do we assume uniformity in the radius $r$ of the balls $B(x,r)$
involved.
It is worth noting
that the corresponding result with
global assumptions and a global conclusion fails in locally compact spaces,
due to a counterexample by Koskela~\cite{Koskela}.
In complete spaces it holds by Keith--Zhong~\cite{KZ}.

In Section~\ref{sect-Leb}, we turn to Lebesgue points and show that
Sobolev (Newtonian) functions have Lebesgue points 
outside a set of zero \p-capacity.
Traditionally,
as well as in metric spaces, such results are shown
using the density of continuous functions.
Here we avoid using this property and
instead exploit the Newtonian theory in a different
and novel way, which may be of interest also under global assumptions.

In the next section we consider the
density of Lipschitz and locally Lipschitz functions
in the Sobolev (Newtonian) space $\Np(X)$.
There are two existing results under global assumptions in the literature,
one assuming doubling and a Poincar\'e inequality, due to Shanmugalingam~\cite{Sh-rev},
and the other more recent one assuming $p>1$, completeness and doubling but no Poincar\'e
inequality, due to 
Ambrosio--Colombo--Di Marino~\cite{AmbCD}
and Ambrosio--Gigli--Savar\'e~\cite{AmbGS}.
We extend both results to local assumptions, and 
combine them.
Among other results, we obtain the following ``local-to-global''
density result.

\begin{thm} \label{thm-Lipc-dense-intro}
If $X$ is proper and connected and $\mu$ is locally doubling and 
supports a local \p-Poincar\'e inequality then 
Lipschitz functions with compact support are dense in $\Np(X)$. 
\end{thm}

In Section~\ref{sect-qcont} we look at consequences of the 
obtained density results,
primarily quasicontinuity and various properties of
the Sobolev capacity $\Cp$.
We end the paper with a discussion on how much of the nonlinear 
potential theory for \p-harmonic functions,
developed e.g.\ in the book \cite{BBbook}, holds under local assumptions,
and explain 
that indeed most of the results therein
extend to
this setting, with the main exception being the Liouville theorem,
which actually fails without global assumptions.
As most of the results in \cite{BBbook} are either local or semilocal 
(e.g.\ on bounded domains) this is
not so surprising, and indeed this has already been hinted upon in the literature, 
as mentioned above.

The importance of distinguishing between local and global assumptions 
is certainly more
apparent when discussing global properties, such as the Dirichlet problem on
unbounded domains 
(as in Hansevi~\cite{Hansevi1}, \cite{Hansevi2}) or
existence of global singular functions
(as in Holopainen--Shanmugalingam~\cite{HoSh}).
We hope that the theory developed in this paper will provide
a suitable foundation for such studies.

\begin{ack}
We would like to thank the anonymous referee for
encouraging us to add more details to some of the proofs.
The authors 
were supported by the Swedish Research Council, 
grants 621-2011-3139, 621-2014-3974 and 2016-03424.
\end{ack}

\section{Upper gradients and Newtonian spaces}
\label{sect-prelim}

We assume throughout the paper
that $1 \le p<\infty$ and that $X=(X,d,\mu)$ is a metric space equipped
with a metric $d$ and a positive complete  Borel  measure $\mu$
such that $0 < \mu(B)<\infty$ for all balls $B \subset X$. 
It follows that $X$ is  separable and
Lindel\"of.
Proofs of the results in 
this section can be found in the monographs
Bj\"orn--Bj\"orn~\cite{BBbook} and
Heinonen--Koskela--Shanmugalingam--Tyson~\cite{HKSTbook}.

A \emph{curve} is a continuous mapping from an interval,
and a \emph{rectifiable} curve is a curve with finite length.
We
will only consider curves which are compact
and 
rectifiable, and thus each curve can 
be parameterized by its arc length $ds$. 
A property is said to hold for \emph{\p-almost every curve}
if it fails only for a curve family $\Ga$ with zero \p-modulus, 
i.e.\ there exists $0\le\rho\in L^p(X)$ such that 
$\int_\ga \rho\,ds=\infty$ for every rectifiable curve $\ga\in\Ga$.

Following Heinonen--Kos\-ke\-la~\cite{HeKo98},
we introduce upper gradients 
as follows 
(they called them very weak gradients).

\begin{deff} \label{deff-ug}
A Borel function $g: X \to [0,\infty]$  is an \emph{upper gradient} 
of a function $f:X \to \eR:=[-\infty,\infty]$
if for all nonconstant rectifiable curves  
$\gamma \colon [0,l_{\gamma}] \to X$,
\begin{equation} \label{ug-cond}
        |f(\gamma(0)) - f(\gamma(l_{\gamma}))| \le \int_{\gamma} g\,ds,
\end{equation}
where the left-hand side is considered to be $\infty$ 
whenever at least one of the 
terms therein is infinite.
If $g:X \to [0,\infty]$ is measurable 
and \eqref{ug-cond} holds for \p-almost every nonconstant rectifiable curve,
then $g$ is a \emph{\p-weak upper gradient} of~$f$. 
\end{deff}

The \p-weak upper gradients were introduced in
Koskela--MacManus~\cite{KoMc}. 
It was also shown therein
that if $g \in \Lploc(X)$ is a \p-weak upper gradient of $f$,
then one can find a sequence $\{g_j\}_{j=1}^\infty$
of upper gradients of $f$ such that $\|g_j-g\|_{L^p(X)} \to 0$.
If $f$ has an upper gradient in $\Lploc(X)$, then
it has an a.e.\ unique \emph{minimal \p-weak upper gradient} $g_f \in \Lploc(X)$
in the sense that for every \p-weak upper gradient $g \in \Lploc(X)$ of $f$ we have
$g_f \le g$ a.e., see Shan\-mu\-ga\-lin\-gam~\cite{Sh-harm}.
Following Shanmugalingam~\cite{Sh-rev}, 
we define a version of Sobolev spaces on the metric space $X$.

\begin{deff} \label{deff-Np}
For a measurable function $f:X\to \eR$, let 
\[
        \|f\|_{\Np(X)} = \biggl( \int_X |f|^p \, d\mu 
                + \inf_g  \int_X g^p \, d\mu \biggr)^{1/p},
\]
where the infimum is taken over all upper gradients $g$ of $f$.
The \emph{Newtonian space} on $X$ is 
\[
        \Np (X) = \{f: \|f\|_{\Np(X)} <\infty \}.
\]
\end{deff}

The quotient
space $\Np(X)/{\sim}$, where  $f \sim h$ if and only if $\|f-h\|_{\Np(X)}=0$,
is a Banach space and a lattice, see Shan\-mu\-ga\-lin\-gam~\cite{Sh-rev}.
In this paper we assume that functions in $\Np(X)$
 are defined everywhere (with values in $\eR$),
not just up to an equivalence class in the corresponding function space.
This is important for upper gradients to make sense.

For a measurable set $E\subset X$, the Newtonian space $\Np(E)$ is defined by
considering $(E,d|_E,\mu|_E)$ as a metric space in its own right.
We say  that $f \in \Nploc(E)$ if
for every $x \in E$ there exists a ball $B_x\ni x$ such that
$f \in \Np(B_x \cap E)$.
If $f,h \in \Nploc(X)$, then $g_f=g_h$ a.e.\ in $\{x \in X : f(x)=h(x)\}$,
in particular  for $c \in \R$ we have
$g_{\min\{f,c\}}=g_f \chi_{\{f < c\}}$ a.e.

\begin{deff}
The (Sobolev) \emph{capacity} 
of a set $E$  is the number 
\begin{equation*} 
  \Cp(E) =\inf_u    \|u\|_{\Np(X)}^p,
\end{equation*}
where the infimum is taken over all $u\in \Np (X)$ such that $u=1$ on $E$.
\end{deff}

We say that a property holds \emph{quasieverywhere} (q.e.)\ 
if the set of points  for which the property does not hold 
has capacity zero. 
The capacity is the correct gauge 
for distinguishing between two Newtonian functions. 
If $u \in \Np(X)$, then $u \sim v$ if and only if $u=v$ q.e.
Moreover, 
if $u,v \in \Nploc(X)$ and $u= v$ a.e., then $u=v$ q.e.

We let $B=B(x,r)=\{y \in X : d(x,y) < r\}$ denote the ball
with centre $x$ and radius $r$, and let $\la B=B(x,\la r)$.
We assume throughout the paper that balls are open.
In metric spaces it can happen that
balls with different centres and/or radii 
denote the same set. 
We will however make the convention that a ball $B$ comes with
a predetermined centre and radius $r_B$. 
Note that it can happen that  $B(x_0,r_0) \subset B(x_1,r_1)$
even when $r_0 > r_1$.
In disconnected spaces this can happen also when $r_0 > 2 r_1$.
If $X$ is connected, then $r_0 >2r_1$ is possible only when 
$B(x_0,r_0)= B(x_1,r_1)=X$.

\section{Local doubling}
\label{sect-doubling}

One can think of several different possibilities for local assumptions. 
We will make them precise below. 
In this section we concentrate on the doubling property and then 
consider Poincar\'e inequalities in the next section.

\begin{deff} \label{def-local-doubl-mu}
The measure \emph{$\mu$ is doubling within $B(x_0,r_0)$}
if there is $C>0$ (depending on $x_0$ and $r_0$)
such that $\mu(2B)\le C \mu(B)$ 
for all balls $B \subset B(x_0,r_0)$.

We say that $\mu$ is \emph{locally doubling} (on $X$) if 
for every $x_0 \in X$ there is $r_0>0$ 
(depending on $x_0$) such that $\mu$ is doubling within $B(x_0,r_0)$.

If $\mu$ is doubling within every ball $B(x_0,r_0)$ then 
it is \emph{semilocally doubling} (on $X$),
and if moreover $C$ is independent of $x_0$ and $r_0$,
then $\mu$ is \emph{globally doubling} (on $X$).
\end{deff}

Note  that when saying that $\mu$ is doubling \emph{within} $B(x_0,r_0)$
this is (implicitly) done with respect to $X$ as the balls are all
with respect to $X$, and moreover $2B$ does not have to be a subset
of $B(x_0,r_0)$.
This is not the same as saying that $\mu$ is globally doubling 
\emph{on} $B(x_0,r_0)$,
which refers to balls with respect to $B(x_0,r_0)$.

If $\mu$ is locally doubling on $X$ and $\Om \subset X$ is open,
then $\mu$  is also locally doubling on $\Om$.
This hereditary property
fails for semilocal and global doubling,
see 
Remark~\ref{rmk-global-restriction} below.

An even weaker property is that $\mu$ is \emph{pointwise doubling} 
at $x_0 \in X$ if there are $C,r_0>0$ such that 
$\mu(B(x_0,2r)) \le C \mu(B(x_0,r))$ for $0<r<r_0$.
Requiring such a pointwise assumption and a similar pointwise Poincar\'e 
inequality at every $x_0 \in X$ is too weak for most results.
See however, Bj\"orn--Bj\"orn--Lehrb\"ack~\cite{BBLeh1}, \cite{BBLehFirstthin}
for  capacity estimates using such
pointwise assumptions.

\begin{deff}  \label{def-local-doubl-X}
The space $X$ is \emph{globally doubling}
if there is a constant $N$ such that every ball $B(x,r)$ 
can be covered by at most $N$
balls with radii $\frac{1}{2}r$. 

The space $X$ is \emph{locally doubling} if 
for every $x_0 \in X$ there is $r_0>0$ 
such that $B(x_0,r_0)$ is globally doubling.
Moreover,  $X$ is 
\emph{semilocally doubling}
if every ball $B \subset X$ is globally doubling.
\end{deff}

\begin{remark} \label{rmk-global-restriction}
Let $B$ be a ball.
If $\mu$ is globally doubling then
it does not follow that $\mu|_B$ is globally doubling
on $B$, see Example~\ref{ex-local-ass-not-inherit}.
On the other hand if the space $X$ is globally doubling then
so is $B$ (as a metric space).
This is why Definition~\ref{def-local-doubl-X} 
differs from Definition~\ref{def-local-doubl-mu}
in that it considers balls with respect to
$B$ which are not necessarily balls with respect to $X$.
It is possible to give an equivalent definition
of (semi)local doubling of $X$ more in the spirit
of Definition~\ref{def-local-doubl-mu}, which only uses
balls with respect to
$X$, but such a definition is 
more technical to state and hence
we prefer our Definition~\ref{def-local-doubl-X}.
\end{remark}

It is rather immediate that every subset of a globally doubling
metric space is itself globally doubling, and hence the same
hereditary property also holds for (semi)local doubling.
It is also easy to see that every bounded set in a semilocally
doubling metric space is totally bounded.
See Heinonen~\cite[Section~10.13]{heinonen} for more on doubling
metric spaces. 

If $\mu$ is (semi)locally resp.\ globally doubling, 
then so is
$X$ by the following result.

\begin{prop} \label{prop-doubling-mu-2/3}
Assume that $\mu$ is doubling within $B_0=B(x_0,r_0)$
in the sense of Definition~\ref{def-local-doubl-mu}.
Then $\de B_0$ is globally doubling for every $\de< \tfrac{2}{3}$,
with $N$ depending only on $\de$ and the doubling constant within $B_0$.
\end{prop}

Example~\ref{ex-doubling-2/3} below shows that the
constant $\tfrac{2}{3}$ is optimal and that
it can even happen that $\tfrac23 B_0$ is not totally bounded.

\begin{proof}
Let $B'=B(x,r) \cap \de B_0$ be an arbitrary ball 
with respect to
$\de B_0$ for some $\de < \tfrac{2}{3}$.
Then $x \in \de B_0$ and we may assume that $r \le 2 \de r_0$.
Let $B=B(x,r)$ and
$r'=\min\bigl\{\tfrac{1}{4}r,\tfrac{1}{24}r_0\bigr\}$.
Assume that $x_i \in B'$, $i=1,\ldots,N$, are such that
$d(x_i,x_j) \ge  2r'$ if $i \ne j$.
Then $B(x_i,r')$ are pairwise disjoint and
$B(x_i,8r') \subset B_0$.
We shall show that there is a bound for $N$.
Let $C_\mu$ be the doubling constant for $\mu$ within $B_0$.

If $r'=\tfrac{1}{4}r$, then 
\begin{equation}   \label{eq-C-mu-4}
   \mu(2B)
   \le \mu(B(x_i,16r'))
  \le C_\mu  \mu(B(x_i,8r'))
  \le C_\mu^4  \mu(B(x_i,r')),
\end{equation}
and hence
\[
   N \min_i \mu(B(x_i,r')) 
   \le \sum_{i=1}^N \mu(B(x_i,r')) 
   \le \mu(2B)
   \le C_\mu^4 \min_i \mu(B(x_i,r')),
\]
which yields that $N \le C_\mu^4$. 
On the other hand, if $r'=\tfrac{1}{24}r_0$ then 
as in~\eqref{eq-C-mu-4},
\[
   \mu\bigl(\bigl(\tfrac{2}{3}-\de\bigr) B_0\bigr) 
  \le  \mu\bigl(B\bigl (x_i,\tfrac{2}{3}r_0\bigr)\bigr) 
  =   \mu\bigl(B(x_i,16r'))
  \le C_\mu^{4}  \mu(B(x_i,r')).
\]
Therefore
\begin{align*}
   N \min_i \mu(B(x_i,r')) 
   & \le \sum_{i=1}^N \mu(B(x_i,r')) 
   \le \mu(B_0) \\
   &\le \frac{C_\mu^{4}\mu(B_0)}{   \mu\bigl(\bigl(\tfrac{2}{3}-\de\bigr) B_0\bigr)}
    \min_i \mu(B(x_i,r'))
   \le M \min_i \mu(B(x_i,r')),
\end{align*}
where $M$ only depends on $C_\mu$ and  $\de$.
Hence $N \le M$.

We can thus find a maximal pairwise  disjoint collection of 
$\{B(x_i,r')\}_{i=1}^N$ with  $N \le \max\{M,C_\mu^4\}$ elements.
As the collection is maximal we see that 
\[
   B' 
\subset \bigcup_{i=1}^N B(x_i,2r') 
\subset \bigcup_{i=1}^N B\bigl(x_i,\tfrac{1}{2}r\bigr).
\] 
Hence $\de B_0$ is globally doubling.
\end{proof}

\begin{example} \label{ex-doubling-2/3}
Let $I_0=\{0\}\times[-1,0]\subset\R^2$ and 
$I_j=\{j\}\times\bigl[\tfrac23-\tfrac13\cdot 2^{-j},1\bigr]\subset\R^2$, 
$j=1,2,\ldots$,
be vertical linear segments in the plane.
Equip each $I_j$ with a multiple of the 1-dimensional Hausdorff measure so
that $\mu(I_j)=2^{-j}$, $j=0,1,\ldots$.
Let $X=\bigcup_{j=0}^\infty I_j$, equipped with $\mu$ and the metric $d$
so that for $x=(j,x_2)\in I_j$ and $y=(k,y_2)\in I_k$,
\[
d(x,y)= \begin{cases}
      |x_2-y_2|, & \text{if } j=k \text{ or } j=0 \text{ or } k=0, \\
      1, & \text{if } j \ne k, \ j,k\ge 1.
      \end{cases}
\]
Let $x_0=(0,0)$. Then it is easily verified that $\mu$ is doubling
within $B_0=B(x_0,1)$. 
However, $\tfrac23 B_0$ is not totally bounded, and thus not globally doubling.

This example also shows that the next two results are sharp.
More precisely, $X$ is bounded and complete but noncompact and thus not
proper.
Both $X$ and $\mu$ are locally doubling but neither is semilocally 
doubling. 
\end{example}

The most common global assumptions are that
$X$ is complete and supports a global \p-Poincar\'e
inequality and that $\mu$ is globally doubling.
It then follows that $X$ is proper and connected (and even quasiconvex).
Under local assumptions these properties
need to be imposed separately.
Connectedness is strongly related to Poincar\'e inequalities,
which we discuss in the next section. 
Properness always implies completeness and 
the proof of \cite[Proposition~3.1]{BBbook} also shows the following
equivalence.

\begin{lem}   \label{lem-proper-equiv-complete}
Assume that $X$ is semilocally doubling.
Then $X$ is proper if and only if $X$ is complete.
\end{lem}

If $X$ is only locally doubling and complete, then it is locally compact
but not necessarily proper as the following example shows.

\begin{example}
Let  $X=\R^2$ equipped with the Gaussian measure 
$Ce^{-x_1^2-x_2^2}\,dx$ but with the
distance $d(x,y)=\arctan |x-y|$.
Then $X$ is bounded and complete, but not proper (as $X$ is a closed bounded 
noncompact set). 
At the same time,
$\mu$ clearly is locally doubling and
  supports a local $1$-Poincar\'e inequality.
\end{example}

In proper spaces, local and semilocal doubling are equivalent,
as we shall now see.

\begin{prop} \label{prop-semilocal-doubling}
Assume that $X$ is proper.
If $X$ resp.\ $\mu$ is locally doubling, 
then it is semilocally doubling.
\end{prop}

\begin{proof}
Let $B_0$ be a ball.
If $X$ is locally doubling,  we can for each $x \in \itoverline{B}_0$
find a globally doubling ball $B_x \ni x$.
Since $X$ is proper, $\itoverline{B}_0$ is compact and thus
 we can find a finite set $\{x_i\}_{i=1}^N$
such that $B_0 \subset \bigcup_{i=1}^N B_{x_i}$.
It is easily seen that a finite union of globally doubling
sets is globally doubling, and hence $B_0$ is globally doubling.

Now assume instead that $\mu$ is locally doubling
and $B_0=B(x_0,r_0)$. 
By enlarging $r_0$, if necessary, we may assume that either
$r_0=\dist(x_0, X \setm B_0)$
or $B_0=X$.
Since $\itoverline{B}_0$ is compact (as $X$ is proper),
it can be covered by finitely many balls $B_j=B(x_j,r_j)$ 
such that $\mu$  is doubling within each ball $2B_j$. 
Let $r'=\min_{j} r_j$.
Again by compactness, we can cover $\itoverline{B}_0$ 
by finitely many balls $B'_j$ with
radii $r'/2$. 
Let $B(x,r) \subset B_0$ be arbitrary
and find $j$ and $j'$ such that $x \in B_j$ and $x \in B'_{j'}$.

If $r\le r'$ then $B(x,r)\subset 2B_j$ and hence the conclusion of the
proposition holds for $B(x,r)$ with constant $\max_{j} C_j$.
On the other hand, if
$r> r'$ then
$x\in B'_{j'}$  and hence
$B'_{j'}\subset B(x,r)$, which yields
the lower bound
\[
\mu(B(x,r))\ge \min_{j} \mu(B'_j)>0.
\]
Since $B(x,2r)\subset B(x_0,5r_0)$, we also have a uniform upper 
bound for $\mu(B(x,2r))$, which proves that $\mu$ is semilocally doubling.
\end{proof}

We will need the following local maximal function estimate.

\begin{prop} \label{prop-maximal-fn}
Assume that $\mu$ is doubling within the ball $B_0$
in the sense of Definition~\ref{def-local-doubl-mu},
and let $\Om \subset B_0$ be open.
For $f\in L^1(\Om)$, define
the noncentred local maximal function
\begin{equation} \label{eq-def-local-max-fn} 
     M^*_{\Om,B_0} f(x) := \sup_{B}
          \vint_{B} f\,d\mu,
          \quad x \in \Om,
\end{equation}
where the supremum is taken over all balls $B$ such that $x \in B \subset \Om$
and $\frac52 B\subset B_0$.
Then
\begin{equation} \label{eq-max-weak-L1} 
\mu(\Etau) \le \frac{C}{\tau}\int_{\Etau}|f|\,d\mu,
\quad \text{where } \Etau=\{x\in \Om: M^*_{\Om,B_0} f(x)>\tau\}.
\end{equation}
Moreover, if $t>1$, then 
\[
      \int_{\Om} (M^*_{\Om,B_0}f)^t \, d\mu \le C_t \int_{\Om} |f|^t \, d\mu.
\]
\end{prop}

The constant $5$ in the factor $\frac52$ above and in the proof below comes
from the $5$-covering lemma. 
It is well known that also the  $(3+\eps)$-covering lemma
holds, for every $\eps>0$, 
see \cite[Remark~1.8, Example~1.9 and p.~36]{BBbook}.
Thus the factor $\frac{5}{2}$ can be replaced by any factor $>\frac{3}{2}$,
which would also make it possible to decrease some other constants in 
this paper. 
For simplicity we have chosen to just rely on the $5$-covering lemma, 
as is common practice in analysis on metric spaces.

\begin{proof}
Since  $\mu$ is doubling within the ball $B_0$,
it is true that $\mu(5B) \le C \mu(B)$ for every ball $B$ used
 in \eqref{eq-def-local-max-fn}.
Thus, the proof of Lemma~3.12 in \cite{BBbook}
directly applies also here showing  the first estimate
\eqref{eq-max-weak-L1}.
The second estimate then follows just as in
the proof of Theorem~3.13 in \cite{BBbook},
with $X$ therein replaced by~$\Om$.
\end{proof}

We end the section by noting the following consequence of local doubling,
which will be useful
later.

\begin{thm} \label{thm-Leb-ae}
\textup{(The Lebesgue differentiation theorem)}
Assume that $\mu$ is locally doubling.
If $f \in L^1\loc(X)$, then a.e.\
point is a Lebesgue point for $f$.
\end{thm}

\begin{proof}
As $X$ is Lindel\"of, we can cover $X$ by balls $\{B_j\}_{j=1}^\infty$
such that $f \in L^1(10B_j)$ and
$\mu$ is doubling within each $10 B_j$.
By the proof of Theorem~1.6 in Heinonen~\cite{heinonen},
the Vitali covering theorem holds in each $B_j$.
It then follows from Remark~1.13 in \cite{heinonen} that
the Lebesgue differentiation theorem holds within each $B_j$,
and hence in $X$, as the union is countable.
\end{proof}

\section{Local Poincar\'e inequalities}
\label{sect-PI}

In this section we study local aspects of the Poincar\'e inequality
similarly to the doubling property in Section~\ref{sect-doubling}.
It will turn out that connectivity plays an important  role here.

\begin{deff} \label{def-PI}
Let $1 \le q < \infty$.
We say that the
\emph{$(q,p)$-Poincar\'e inequality holds within $B(x_0,r_0)$} 
if there are constants $C>0$ and $\lambda \ge 1$ (depending on $x_0$ and $r_0$)
such that for all balls $B\subset B(x_0,r_0)$, 
all integrable functions $u$ on $\la B$, and all upper gradients $g$ of $u$, 
\begin{equation}    \label{eq-def-local-PI}
        \biggl(\vint_{B} |u-u_B|^q \,d\mu\biggr)^{1/q}
        \le C r_B \biggl( \vint_{\lambda B} g^{p} \,d\mu \biggr)^{1/p}.
\end{equation}
We also say that $X$ (or $\mu$) supports 
a \emph{local $(q,p)$-Poincar\'e inequality} (on $X$) if
for every $x_0 \in X$ there is $r_0$ (depending on $x_0$) 
such that the $(q,p)$-Poincar\'e inequality holds within $B(x_0,r_0)$.

If the $(q,p)$-Poincar\'e inequality holds within every ball $B(x_0,r_0)$
then $X$ supports a \emph{semilocal $(q,p)$-Poincar\'e inequality},
and if moreover $C$ and $\la$ are independent of $x_0$ and $r_0$,
then $X$ supports a \emph{global $(q,p)$-Poincar\'e inequality}.

If $q=1$ we usually just write 
\emph{\p-Poincar\'e inequality}.
\end{deff}

The inequality \eqref{eq-def-local-PI} can equivalently
be required 
for 
all integrable functions $u$ on $\la B$, 
and all \p-weak upper gradients $g$ of $u$;
see \cite[Proposition~4.13]{BBbook} for other equivalent formulations.

As in the case of the
doubling condition, 
local Poincar\'e inequalities
  are inherited by open subsets, i.e.\ 
if $\Om \subset X$ is open and
$X$ supports a local $(q,p)$-Poincar\'e inequality,
then so does $\Om$.
This hereditary property
fails for semilocal and global Poincar\'e inequalities.

\begin{remark}
When defining (semi)local doubling 
in Definition~\ref{def-local-doubl-mu}
it is primarily a matter of taste 
(except for the constant $\tfrac23$ in Proposition~\ref{prop-doubling-mu-2/3})
whether the condition
is required for $B \subset B(x_0,r_0)$
or for $B \subset 2B \subset B(x_0,r_0)$, and the same is
true for $B$ and $\la B$ in the local Poincar\'e inequalities
in Definition~\ref{def-PI}.
However, for semilocal Poincar\'e inequalities it is 
 vital that the condition is for all $B \subset B_0$,
rather than 
for all $B \subset \la B \subset B_0$, which is a weaker requirement 
since $\la$ is allowed to depend on $B_0$.
(Consider e.g.\ $X=(-\infty,0]\cup[1,\infty)$ with the Euclidean metric
and Lebesgue measure.
Then for $x_0=0$, $r_0\ge2$ and $\la:=r_0$, the requirement $\la B\subset B(x_0,r_0)$
implies that $r_B\le1$ and hence 
\eqref{eq-def-local-PI} holds for all such balls,
while it clearly fails for $B(0,2)\subset B(x_0,r_0)$.)
\end{remark}

\begin{example} \label{ex-local-ass-not-inherit}
If a Poincar\'e inequality holds within
$B_0=B(x_0,r_0)$, 
then
it does not mean that $B_0$ (or $\itoverline{B}_0$) 
itself 
(as a metric space) supports a global 
Poincar\'e inequality.
Similarly, if 
$\mu$ is doubling within $B_0$, 
then it does not follow that  $\mu|_{B_0}$ (or $\mu|_{\itoverline{B}_0}$) 
is doubling.
To see this, 
let 
\[
X=\R^2 \setm \bigl\{(x,y)\in\R^2: 0<y<\sqrt{1-x^2}-e^{-1/|x|}\bigr\},
\]
i.e.\ $X$ is $\R^2$ with the open unit upper half-disc removed
and the curved cusp of exponential type at $(0,1)$,
\[
C_0=\bigl\{(x,y)\in\R^2: 0<\sqrt{1-x^2}-e^{-1/|x|}\le y<\sqrt{1-x^2}\bigr\},
\]
inserted back into the hole.
Then $\interior X$ is a uniform domain and hence the Lebesgue measure,
restricted to it,
is globally doubling and supports a global $1$-Poincar\'e inequality,
by Theorem~4.4 in Bj\"orn--Shanmugalingam~\cite{BjShJMAA}
(or \cite[Theorem~A.21]{BBbook}).
By Aikawa--Shanmugalingam~\cite[Proposition~7.1]{AikSh05},
the same is true for $X$ itself.

Now, if $x_0$ is the origin and we let $B_0=B(x_0,1)$, then $B_0\cap X$
and $\itoverline{B_0\cap X}$ have the cusps $C_0$ and $\itoverline{C}_0$ 
in the vicinity of the point $(0,1)$.
Hence the Lebesgue measure restricted to these sets is not doubling
near or at this point.

Moreover, $C_0$ is disconnected at $(0,1)$ and $\itoverline{C}_0$
is essentially disconnected at the point $(0,1)$
(which has zero capacity
with respect to $\itoverline{B_0\cap X}$ for all $p\ge 1$), 
so neither $B_0\cap X$ nor $\itoverline{B_0\cap X}$ supports any global
or semilocal Poincar\'e inequalities.
For $\itoverline{B_0\cap X}$ even the local doubling and 
all local Poincar\'e inequalities fail at $(0,1)$. 
\end{example}

Propositions~\ref{prop-semilocal-doubling-intro} 
and~\ref{prop-semilocal-doubling} about (semi)local
doubling are rather straightforward.
A bit more surprising, perhaps, is that 
Theorem~\ref{thm-PI-intro}
is true for Poincar\'e inequalities.
We will obtain the following more general version of Theorem~\ref{thm-PI-intro}.

\begin{thm} \label{thm-PI}
If $X$ is proper and connected and $\mu$ is locally doubling
and supports a local $(q,p)$-Poincar\'e inequality, then 
it supports a semilocal $(q,p)$-Poincar\'e inequality.
\end{thm}

The proof of Theorem~\ref{thm-PI} will be split into a number
  of lemmas, some of which may be of independent interest.
  It will be concluded  at the end of this section.
But first, we
discuss the assumptions in Theorem~\ref{thm-PI} as well as some
consequences of local Poincar\'e inequalities.
To start with, it is easily verified that if $X$ supports a semilocal 
\p-Poincar\'e inequality then it is connected, cf.\ the proof of 
\cite[Proposition~4.2]{BBbook}.
The following example shows that this conclusion fails if we replace the
semilocal assumption with a local one, even if $X$ is 
proper.

\begin{example}
Let $X$ be the union of two disjoint closed balls in $\R^n$,
which is proper and 
such that the Lebesgue measure is globally doubling and
supports a local $1$-Poincar\'e inequality on $X$.
However $X$  is not connected.
\end{example}

The above example also shows that the connectedness assumption in 
Theorem~\ref{thm-PI}
cannot be dropped, while the following example shows that neither can the properness
assumption. 

\begin{example}
Let
\[
  X=(\itoverline{B(0,2)} \setm B(0,1))
    \cup \{x=(x_1,x_2): 0 < |x| \le  2 \text{ and } x_1x_2 \ge 0\}
    \subset \R^2.
\]
Then $X$ is connected and the Lebesgue measure 
is globally doubling on $X$ and supports
a local $1$-Poincar\'e inequality.
However, $X$ is neither proper nor 
supports any 
semilocal Poincar\'e inequality.
\end{example}

The following is a partial result on the way to proving
Theorem~\ref{thm-PI}.

\begin{lem}  \label{lem-PI-small-r-OK}
If $X$ is proper and supports a local $(q,p)$-Poincar\'e inequality, then 
for every
ball $B_0$ there exist $\rp,C>0$ and $\la \ge 1$
such that the $(q,p)$-Poincar\'e inequality~\eqref{eq-def-local-PI}
holds for all balls 
$B=B(x,r)\subset B_0$ with $r\le \rp$.
\end{lem}

\begin{proof}
As in the proof of Proposition~\ref{prop-semilocal-doubling},
the compact set $\itoverline{B}_0$ can be covered by finitely many balls
$B_j=B(x_j,r_j)$ so that the $(q,p)$-Poincar\'e inequality
holds within each $2B_j$, 
with constants $C_j$ and $\la_j$.
Letting
\[
\rp=\min_{j} r_j, \quad C=\max_{j} C_j \quad \text{and}
\quad  \la=\max_{j} \la_j,
\]
together with the fact that $B(x,r)\subset2B_j$ for some $j$,
concludes the proof.
\end{proof}

We shall now see which connectivity properties follow from the
local Poincar\'e inequality.

\begin{prop}   \label{prop-ex-curve-in-B0}
Assume that $X$ is locally compact
and that $\mu$ is locally doubling and 
supports a local \p-Poincar\'e inequality.
Then $X$ is locally quasiconvex {\rm(}and thus locally
rectifiably pathconnected\/{\rm)}, 
i.e.\ for every $x_0\in X$ there are $r_0, L>0$ 
{\rm(}depending on $x_0${\rm)} such that every pair of 
points $x,y\in B(x_0,r_0)$ can be connected by 
a curve of length at most $L d(x,y)$.

If $X$ is moreover proper and connected, then it
is semilocally quasiconvex, i.e.\ the above connectivity property holds
in every ball $B(x_0,r_0)$ {\rm(}with $L$ depending on it\/{\rm)}.
\end{prop}

Local quasiconvexity and its behaviour under various transformations
of $X$ were considered by Buckley--Herron--Xie~\cite{BuckleyHerronXie}, cf.\ 
Section~2 and Proposition~4.2 therein. 
Note that if $X$ is connected and locally (rectifiably) pathconnected,
then it is (rectifiably) pathconnected, as the (rectifiably)
pathconnected components must be open.
In particular, local quasiconvexity and connectedness
imply that $X$
is rectifiably pathconnected, but not necessarily quasiconvex.

Later on it will be important 
to have the following more precise
version of the above connectivity result, which will also be used 
to deduce Proposition~\ref{prop-ex-curve-in-B0}.

\begin{lem}   \label{lem-ex-curve-in-B0}
Let $x,y\in X$ and assume that 
the \p-Poincar\'e inequality and the 
doubling property for $\mu$ hold
{\rm(}with constants $\CPI$ and $C_\mu${\rm)} within $B_0=B(x,2d(x,y))$
in the sense of Definitions~\ref{def-local-doubl-mu} and~\ref{def-PI}.
Let $\La=3C_\mu^3\CPI$. 
If the ball $\itoverline{\La B}_0$ is compact then $x$ and $y$  
can be connected 
by a
curve in $\itoverline{\La B}_0$,
of length at most $L  d(x,y)$, where $L=9\Lambda$. 
\end{lem}

Note that, as in \cite[Theorem~4.32]{BBbook}, the constants 
$\La$ and $L$ are independent of the dilation constant $\la$ 
in the \p-Poincar\'e inequality.

\begin{proof}
Let $\la$ be the dilation constant in the \p-Poincar\'e inequality
within $B_0$.
Following Semmes's chaining argument, define
for $\eps>0$ and $z\in \la B_0$, 
\[
\rho_\eps(z)=\inf \sum_{i=1}^m d(x_{i-1},x_i),
\]
where
the infimum is taken over all collections $\{x_i\}_{i=0}^m\subset X$ 
such that
$x_0=x$, $x_m=z$ and $d(x_{i-1},x_i)<\eps$ for all $i=1,2,\ldots,m$.
Should there be no such chain, we let 
$\rho_\eps(z)=10\La d(x,y)$.
Then it is easily verified that
$\rho_\eps$ is locally 1-Lipschitz and has 1 as an upper gradient.

Since the \p-Poincar\'e inequality and the 
doubling property for $\mu$ hold for all the balls $B_0=B(x,2d(x,y))$ 
and $B_j=B(y,2^{-j}d(x,y))\subset B_0$, with $j=1,2,\ldots$,
a standard telescoping argument 
shows that
\begin{align} \label{eq-std-telescope}
|\rho_\eps(y)-(\rho_\eps)_{B_0}|
&
\le \sum_{j=0}^\infty |(\rho_\eps)_{B_{j+1}}-(\rho_\eps)_{B_j}|
\nonumber
\\ &
\le \sum_{j=0}^\infty \vint_{B_{j+1}}
        |\rho_\eps-(\rho_\eps)_{B_j}|\,d\mu \\
&\le \sum_{j=0}^\infty C_\mu^3 \vint_{B_j}
        |\rho_\eps-(\rho_\eps)_{B_j}|\,d\mu 
\nonumber \\ & \le C_\mu^3 \CPI \sum_{j=0}^\infty r_{B_j} 
        \biggl( \vint_{\la B_j} 1^p \,d\mu \biggr)^{1/p}
\nonumber \\ &= \La d(x,y). \nonumber
\end{align}
A similar estimate with $\rho_\eps(x)=0$ then yields that
\[
\rho_\eps(y) = |\rho_\eps(y)-\rho_\eps(x)| \le 2\La d(x,y) < 10\La d(x,y).
\]
In particular, for each $\eps_n = 3\cdot 2^{-n}\La d(x,y)$, $n=1,2,\ldots$,
there is a chain
$x=x_0^n, x_1^n, \ldots, x_{M_n}^n=y$ such that $d(x_{i-1},x_i)\le\eps_n$
for all $i$ and
\[
\sum_{i=1}^{M_n} d(x_{i-1}^n,x_i^n)\le 3\La d(x,y). 
\]
Moreover, as $d(x,x_i^n)\le 2\La d(x,y)=\La r_{B_0}$ 
or $d(y,x_i^n)\le \La d(x,y)$,
we conclude that all $x_i^n$ belong to the compact set $\itoverline{\La B}_0$.

Using \cite[Lemma~4.34]{BBbook}, we can find a subchain
\[
x=\xh_0^n, \xh_1^n, \ldots, \xh_{m_n}^n=y,
\quad \text{satisfying }\eps_n \le d(\xh_{i-1}^n,\xh_i^n)\le 3\eps_n,
\]
which implies that $m_n\le 3\La d(x,y)/\eps_n = 2^n$. 
Letting $S_n=\{2^{-n}i: i=0,1,\ldots,2^n\}$, 
the function $\gamma_n:S_n\to \itoverline{\La B}_0$,
defined by $\gamma_n(2^{-n}i)=\xh^n_{\min\{i,m_n\}}$, is easily shown to be
$9\La d(x,y)$-Lipschitz.

Using a diagonal argument, we can from the sequence 
$\{\gamma_n\}_{n=1}^\infty$ choose 
a subsequence, which converges on $S=\bigcup_{n=1}^\infty S_n$. 
Since each $\gamma_n$ is $9\La d(x,y)$-Lipschitz, so is the limiting
function $\gamma:S\to \itoverline{\La B}_0$.
Finally, $\gamma$ extends to a 
$9\La d(x,y)$-Lipschitz function on $[0,1]$, 
which after reparameterization provides us with the desired curve.
\end{proof}

\begin{proof}[Proof of Proposition~\ref{prop-ex-curve-in-B0}]
Let $x_0\in X$ and find a ball $B_0'$ 
centred at $x_0$
so that $\clBzeroprime$
is compact  and the \p-Poincar\'e inequality and the 
doubling property for $\mu$ hold within $B_0'$,
with constants $\CPI$  and $C_\mu$.
Let $\La=3C_\mu^3\CPI$ and $B_0=\La^{-1}B_0'$.
Now, if  $x,y\in \tfrac15 B_0$ then $B(x,2d(x,y))\subset B_0$ and 
$\itoverline{\La B}_0$ is compact.
Hence Lemma~\ref{lem-ex-curve-in-B0} shows
the existence of a connecting curve of length at most 
$9\La d(x,y)$.
Thus $X$ is locally quasiconvex, 
which proves the first part of the proposition.

Now assume that $X$ is in addition proper 
and connected. Let $B_0=B(x_0,r_0)$ 
be arbitrary.  
By Proposition~\ref{prop-semilocal-doubling}, $\mu$
is semilocally doubling.
Furthermore, Lemma~\ref{lem-PI-small-r-OK} implies that 
for some $0<r'\le r_0$, the \p-Poincar\'e 
inequality~\eqref{eq-def-local-PI-intro} holds for all $B=B(x,r)\subset 5B_0$
with $0<r\le 2r'$.

Since $\overline{5\La B}_0$ is compact, by the properness of $X$,
Lemma~\ref{lem-ex-curve-in-B0} 
yields that every pair of points  $x,y\in\itoverline{B}_0$ 
with $d(x,y)\le r'$,
can be connected by a curve of length at most $L'd(x,y)$, 
where $L'$ depends only on the doubling and Poincar\'e constants 
provided for $5B_0$ by 
Proposition~\ref{prop-semilocal-doubling-intro} and 
Lemma~\ref{lem-PI-small-r-OK}.

By compactness, 
$\itoverline{B}_0$ can be covered by finitely many
balls $B(x_k,r')$ with 
$x_k\in\itoverline{B}_0$, $k=1,2,\ldots, N$.
As $X$ is connected
and locally quasiconvex, it is rectifiably connected.
Hence, there is for each pair $x_j,x_k$, $j\ne k$,
a rectifiable curve $\ga_{j,k}\subset X$
connecting $x_j$ to $x_k$.
Since there are only finitely many such pairs, it follows that
$M:=\sup_{j,k} l_{\ga_{j,k}} < \infty$.

Finally, let $x,y \in B_0$ be arbitrary.
If $d(x,y) \le r'$, then we already know
that $x$ and $y$ can be connected by a curve of length 
at most $L'd(x,y)$.
If $d(x,y) > r'$, then $x$ and $y$ can be connected to some $x_j$ and 
$x_k$, respectively,
by curves of lengths at most $L'r'$.
Adding $\ga_{j,k}$ to these curves produces a curve from $x$ to $y$
of length at most $2L'r' + M < (2L'+M/r')d(x,y)$.
\end{proof}

The local quasiconvexity proved in 
Proposition~\ref{prop-ex-curve-in-B0} can be further bootstrapped 
by the following result.

\begin{lem}   \label{lem-rect-components}
Let $B_0$ be a ball 
such that $\mu$ is locally doubling 
and supports a local \p-Poincar\'e inequality 
on $B_0$ {\rm(}as the underlying space\/{\rm)}.
Assume that the \p-Poincar\'e inequality
\eqref{eq-def-local-PI-intro} holds for $B_0$
\textup{(}in place of $B$\textup{)}
with dilation constant $\la\ge1$ and that $\la B_0$ is locally compact.

Then $B_0$ is rectifiably connected within $\la B_0$, i.e.\
any two points $x,y\in B_0$ can be connected by a 
rectifiable 
curve
lying within $\la B_0$.
\end{lem}

\begin{proof}
Divide $\la B_0$ into its rectifiable components, i.e.\ $x,y \in \la B_0$
belong to the same rectifiable component if there is a 
rectifiable curve
within $\la B_0$ from $x$ to $y$.
For each $x\in\la B_0$, let $G_x$ denote the 
rectifiable component containing 
$x$, which is measurable by 
J\"arvenp\"a\"a--J\"arvenp\"a\"a--Rogovin--Rogovin--%
Shan\-mu\-ga\-lin\-gam~\cite[Corollary~1.9 and Remark~3.1]{JJRRS},
since $\la B_0$ is locally compact.
The local assumptions on $B_0$, together with 
Proposition~\ref{prop-ex-curve-in-B0}, imply that
the sets $G_x\cap B_0$ are open for all $x$.

Let $G$ be such a rectifiable component intersecting $B_0$.
We shall show that $B_0 \subset G$.
Assume on the contrary that $B_0\setm G \ne \emptyset$.
Since all $G_x\cap B_0$ are open, 
so is $B_0\setm G=\bigcup_{x\notin G} (G_x\cap B_0)$.
Let $u=\chi_{G}$, which has $g \equiv 0$ as an upper gradient in 
the open set $\la B_0$
as there are no rectifiable
curves in $\la B_0$ 
from $G$ to 
$\la B_0 \setm G$.
Since both $B_0 \cap G$ and $B_0\setm G$ 
are nonempty and open, they both have positive measure.
Hence, by the \p-Poincar\'e inequality for $B_0$, 
\begin{equation} \label{eq-PI-contradiction}
     0 <    \vint_{B_0} |u-u_{B_0}| \,d\mu
        \le C r_{B_0} \biggl( \vint_{\lambda B_0} g^p  \,d\mu \biggr)^{1/p}
        = 0,
\end{equation}
a contradiction.
\end{proof}

The following lemma makes it possible to lift the Poincar\'e
inequality from  small to large sets.

\begin{lem}  \label{lem-iteration-with-Q}
Let $1 \le q < \infty$ and $A,E\subset X$ be such that 
$\mu(A\cap E)\ge \theta \mu(E)$
for some $\theta>0$.
Also assume that for some $Q\ge0$ and a measurable function $u$,
\[ 
\|u-u_A\|_{L^q(A)} \le Q
\quad \text{and} \quad
\|u-u_E\|_{L^q(E)} \le Q.
\] 
Then
\[
\|u-u_{A\cup E}\|_{L^q(A\cup E)} \le 4(1+\theta^{-1/q})Q.
\]
\end{lem}

\begin{proof}
To start with, we have by the triangle inequality,
\begin{align*}
|u_A-u_E| &= 
\frac{\|u_A - u_E\|_{L^q(A \cap E)}}{\mu(A\cap E)^{1/q}}
\\
&\le \frac{ \|u-u_A\|_{L^q(A)} + \|u-u_E\|_{L^q(E)}}{\mu(A\cap E)^{1/q}} 
\le \frac{2Q}{\mu(A\cap E)^{1/q}}.
\end{align*}
This yields
\begin{align*}
\|u-u_A\|_{L^q(A \cup E)} &\le \|u-u_A\|_{L^q(A)} + \|u-u_E\|_{L^q(E)}
+ \mu(E)^{1/q} |u_A-u_E| \\
& \le 2Q + 2Q \biggl( \frac{\mu(E)}{\mu(A\cap E)}\biggr)^{1/q}
  \le 2 (1+ \theta^{-1/q}) Q.
\end{align*}
Finally, Lemma~4.17 in \cite{BBbook}
allows us to replace $u_A$ on the left-hand
side by $u_{A\cup E}$, at the cost of an extra factor 2 
on the right-hand side.
\end{proof}

We are now ready to conclude the proof of
the semilocal $(q,p)$-Poincar\'e inequality.

\begin{proof}[Proof of Theorem~\ref{thm-PI}]
Let $B_0=B(x_0,r_0)$ be fixed. 
As $X$ is proper and connected,
Proposition~\ref{prop-ex-curve-in-B0}
shows that there is $\sigma \ge 1$
such that every pair of  points in $\itoverline{B}_0$ 
can be connected by a rectifiable curve within $\sigma B_0$.
By Lemma~\ref{lem-PI-small-r-OK}, there exist 
$\rp, C>0$ and $\la\ge 1$ such that
the $(q,p)$-Poincar\'e inequality~\eqref{eq-def-local-PI}
holds for every ball
$B$ with radius $r_B\le \rp$ and centre in $\overline{\sigma B}_0$.
By decreasing $\rp$, if necessary, we may assume that $\la \rp \le r_0$.

Next, let $B\subset B_0$ be an arbitrary ball.
If $r_B\le \rp$, then the 
$(q,p)$-Poincar\'e inequality~\eqref{eq-def-local-PI}
holds for $B$.
Assume therefore that $r_B>\rp$. 
By compactness, $\itoverline{B}_0$ can be covered by finitely many balls $B'_j$,
$j=1,2,\ldots,M$, with radius $\rp$ and 
centres $x_j\in \itoverline{B}_0$.
Proposition~\ref{prop-ex-curve-in-B0} provides us for each $j$
with a rectifiable curve $\ga_j$ in $\sigma B_0$ connecting $x_j$ to $x_{j+1}$.
Following this curve, we define balls 
\[
B_{j,k}=B(\ga_j(k\rp/2),\rp), \quad k=1,2,\ldots  \le \frac{2l_{\ga_j}}{r'}.
\]
Note that $\tfrac12 B_{j,k+1}\subset B_{j,k}$.
Add all these balls to the chain $\{B'_j\}_{j=1}^M$, 
in between $B'_j$ and $B'_{j+1}$,
and renumber the sequence as $B_1, B_2, \ldots, B_N$.
Note that all the balls $B_j$ have the same radius $\rp$ and that
$\tfrac12 B_{j+1} \subset B_j\cap B_{j+1}$.

Next, let $A_j=\bigcup_{i=1}^j B_i$.
Note that $B_0\subset A_N$.
Then $A_j\cap B_{j+1} \supset \tfrac12 B_{j+1}$ and the semilocal
doubling property of $\mu$, provided by
Proposition~\ref{prop-semilocal-doubling-intro}, implies
that for some $\theta>0$ (depending on $2 \sigma B_0$),
\[
\mu(A_j\cap B_{j+1}) \ge \mu\bigl(\tfrac12 B_{j+1}\bigr) 
\ge \theta \mu(B_{j+1}).
\]
Since all the balls $B_j$ have radius $\rp$, the 
$(q,p)$-Poincar\'e inequality~\eqref{eq-def-local-PI} holds for them, 
and we have 
\begin{align*}
 \biggl( \int_{B_j} |u-u_{B_j}|^q \,d\mu \biggr)^{1/q} 
  & \le C\rp \mu(B_j)^{1/q}  \biggl( \vint_{\la B_j} g_u^p\,d\mu \biggr)^{1/p}\\ 
&\le C'\rp \mu(2\sigma B_0)^{1/q}\biggl( \vint_{2\sigma B_0} g_u^p\,d\mu \biggr)^{1/p}
=: Q,
\end{align*}
where we have used the semilocal doubling property, together with
the fact that $\la B_j\subset 2\sigma B_0$ and that $\rp$ and $r_0$
are fixed.
Lemma~\ref{lem-iteration-with-Q} with $A=A_1$ and $E=B_2$ now yields
\[
 \biggl( \int_{A_2} |u-u_{A_2}|^q \,d\mu \biggr)^{1/q} 
\le 4(1+\theta^{-1/q}) 
Q=:\ga Q.
\]
Another application of Lemma~\ref{lem-iteration-with-Q} with 
$A=A_2$ and $E=B_3$ then gives
\[
\biggl( \int_{A_3} |u-u_{A_3}|^q \,d\mu \biggr)^{1/q} 
\le \ga^2 Q.
\]
Continuing in this way along the whole sequence $\{B_j\}_{j=1}^N$, 
we can  conclude after finitely many iterations that 
\begin{align*}
 \biggl( \int_B |u-u_{A_N}|^q \,d\mu \biggr)^{1/q} 
&\le \biggl( \int_{A_N} |u-u_{A_N}|^q \,d\mu \biggr)^{1/q} \\
&\le \ga^N Q  
   = C'' \rp \mu(2\sigma B_0)^{1/q} 
    \biggl( \vint_{2\sigma B_0} g_u^p\,d\mu \biggr)^{1/p}.
\end{align*}
Let $\la'=3\sigma r_0/r'$.
Since $r_B\ge \rp$, the measures of the balls
$B \subset 2\sigma B_0 \subset \la' B$ 
are all comparable, due to the semilocal doubling property of $\mu$,
and thus
\[
 \biggl( \vint_B |u-u_{A_N}|^q \,d\mu \biggr)^{1/q} 
\le  C''' r_B
    \biggl( \vint_{\la' B} g_u^p\,d\mu \biggr)^{1/p}.
\]
Finally, \cite[Lemma~4.17]{BBbook} allows us to replace
$u_{A_N}$ by $u_B$ on the left-hand side
(at the cost of an extra factor 2 on the right-hand side),
which completes the proof.
\end{proof}

We will also need the following lemma, which shows
the reverse doubling condition under suitable local assumptions.

\begin{lem} \label{lem-rev-doubl}
Assume that 
the doubling property for $\mu$ holds within $B_0$ in the sense of
Definition~\ref{def-local-doubl-mu}, 
with doubling constant $C_\mu$,
and that 
the \p-Poincar\'e inequality
\eqref{eq-def-local-PI-intro} holds for $B_0$ 
\textup{(}in place of $B$\textup{)}.
Let $B \subset 2B \subset B_0$ be a ball with 
$r_B < \frac{2}{3} \diam B_0$.
Then there is $\theta<1$, only depending on $C_\mu$,
such that
\[
      \mu\bigl(\tfrac{1}{2}B\bigr) \le \theta \mu(B).
\]
\end{lem}

\begin{proof}
Assume that $B=B(x,r)$.
If there were no $y$ such that $d(x,y)=\frac{3}{4}r$, then
zero would be an upper gradient of $\chi_{\frac{3}{4}B}$,
which
 would violate the \p-Poincar\'e inequality for the ball $B_0$,
as in \eqref{eq-PI-contradiction},
since $\diam \overline{\tfrac{3}{4}B} \le \tfrac{3}{2}r < \diam B_0$ and thus
$B_0 \setm \overline{\tfrac{3}{4}B} \ne \emptyset$.
Hence there is $y$ such that $d(x,y)=\frac{3}{4}r$.

As $B(y,r) \subset 2B \subset B_0$, we get 
from the doubling property within $B_0$ that
\[
    \mu\bigl(\tfrac{1}{2}B\bigr) \le \mu(B(y,2r)) 
    \le C_\mu^3 \mu\bigl(B\bigl(y,\tfrac{1}{4}r\bigr)\bigr)
\]
and thus
\[
     \mu(B) \ge \mu\bigl(\tfrac{1}{2}B\bigr) + \mu\bigl(B\bigl(y,\tfrac{1}{4}r\bigr)\bigr)
          \ge (1+C_\mu^{-3}) \mu\bigl(\tfrac{1}{2}B\bigr).
          \qedhere
\]
\end{proof}

\section{Better Poincar\'e inequalities}
\label{sect-better-PI}

There are two types of better Poincar\'e inequalities.
The first type of result is the Sobolev--Poincar\'e inequality,
due to Haj\l asz--Koskela~\cite{HaKo-CR}, 
\cite{HaKo}, which 
strengthens the left-hand side of the inequality.
The arguments in \cite{HaKo-CR} and \cite{HaKo} also show how,
in sufficiently nice spaces such as $\R^n$, the dilation constant $\la>1$
in the Poincar\'e inequality can be improved to $\la=1$.
These results were originally proved under global doubling and Poincar\'e 
assumptions,
but since all the considerations in the proof are of local nature,
they
can also be obtained 
under (semi)local assumptions in the following form. 

\begin{thm}   \label{thm-(q,p)-PI}
\textup{(Local Sobolev--Poincar\'e inequality)}
Let $B_0$ be a ball such that the \p-Poincar\'e inequality\/ 
\textup{(}with dilation constant $\la$\textup{)} and the
doubling property for $\mu$ hold within $B_0$ in the sense of
Definitions~\ref{def-local-doubl-mu} and~\ref{def-PI}.
Assume that the dimension condition
\begin{equation}    \label{eq-local-dim-cond}
\frac{\mu(B')}{\mu(B)} \ge C_0 \Bigl( \frac{r_{B'}}{r_B} \Bigr)^s 
\end{equation}
holds for some $C_0,s>0$ and all balls $X\ne B'\subset B \subset B_0$.

Then there exists $C$, depending only on $p$, the doubling 
constant and both constants in the \p-Poincar\'e inequality within $B_0$, 
such that for all balls $B$ with $5\la B \subset B_0$, 
all integrable functions $u$ on $2\la B$, and all 
\p-weak upper gradients $g$ of~$u$, 
\begin{equation}    \label{eq-(q,p-PI-2la}
        \biggl(\vint_{B} |u-u_B|^{q} \,d\mu\biggr)^{1/q}
        \le C r_B \biggl( \vint_{2\la  B} g^{p} \,d\mu \biggr)^{1/p},
\end{equation}
where $q=p^*:=sp/(s-p)>p$ if $p < s$ 
while $q<\infty$ is arbitrary when $p\ge s$
\textup{(}in which case $C$ depends also on $q$\textup{)}. 

If $L$ is a local quasiconvexity constant for $B_0$, in the sense that
every pair of points $x,y\in B_0$ can be connected {\rm(}in $X${\rm)}
by a curve of length at most $Ld(x,y)$, then~\eqref{eq-(q,p-PI-2la} holds
for all balls $B$ with $\tfrac52 LB \subset B_0$, 
and the dilation constant $2\la$
in~\eqref{eq-(q,p-PI-2la} can be replaced by $L$.

In particular, if $\mu$ is\/ \textup{(}semi\/\textup{)}locally doubling and 
supports a\/ \textup{(}semi\/\textup{)}local \p-Poincar\'e inequality then
$X$ supports a\/ \textup{(}semi\/\textup{)}local $(p,p)$-Poincar\'e inequality.
\end{thm}

\begin{proof}
It suffices to consider the case $s>p$ since when $s\le p$, one can
instead assume the dimension condition~\eqref{eq-local-dim-cond} 
with $s$ replaced by $p+\eps$  for arbitrarily small $\eps>0$.

We will follow the arguments in the proof of \cite[Theorem~4.39]{BBbook} and
prove~\eqref{eq-(q,p-PI-2la} under the assumption of local
$L$-quasiconvexity, i.e.\ with $LB$ in the right-hand side.
To obtain~\eqref{eq-(q,p-PI-2la} with $2\la B$ in the right-hand side, 
replace the balls in the chain below
by the balls $B^{0,0}:=B(x,2r)$, 
$B^{i,0}:=B(x',2^{-(i+1)}r)$, $i=1,2,\ldots$,
and $\widehat{B}^{i}:=\la B^{i,0}$.
Note that $\frac52 \widehat{B}^{i} \subset  B_0$, $i=0,1,\ldots$.

Let $B\bigl(x,\tfrac52Lr\bigr)\subset B_0$ be arbitrary.
We can assume that $Lr\le\diam B_0$.
Let $\rho_0=Lr/2\la$ and $\rho_i=2^{-i} \rho_0$, $i=1,2,\ldots$\,.
For $x'\in B(x,r)$, consider an $L$-quasiconvex curve $\ga$ from $x$ 
to $x'$, i.e.\ $\ga(0)=x$ and $\ga(l_\ga)=x'$, where $l_\ga\le Ld(x,x')$
is the length of $\ga$.
Find the smallest integer $i'\ge0$ such that $2\la\rho_{i'}\le L(r-d(x,x'))$.
For each $i=0,1,\ldots,i'-1$, consider all the integers $j\ge0$, such that
\[
\tij:= (1-2^{-i})l_\ga+ j\rho_i < (1-2^{-(i+1)})l_\ga,
\]
and let $\xij=\ga(\tij)$.
There are at most $\la$ such $\xij$'s for each $i$.
Similarly, for $i=i'$ there are at most $2\la$ integers
$j\ge0$ and points $\xij=\ga(\tij)$ such that $\tij<l_\ga$.
For $i>i'$, let $j=0$ and $\xij=x'$.

It is now easily verified that $d(x,\xij)+\la\rho_i<Lr$
and hence 
\[
\Bij:=B(\xij,\rho_i)\subset \la \Bij \subset B(x,Lr).
\]

Ordering the balls $\Bij$ lexicographically, we obtain a chain of balls
from $x$ to $x'$, with substantial overlaps.
Assuming that $x'$ is a Lebesgue point of $u$,
a standard telescoping argument using the \p-Poincar\'e 
inequality
for each $\Bij\subset B_0$, as in \eqref{eq-std-telescope},
then yields the estimate
\begin{equation} \label{eq-PI-appl}
|u(x')-u_{B^{0,0}}| 
   \le C_1 \sum_B r_B \biggl( \vint_{\la B} g_u^p\,d\mu \biggr)^{1/p},
\end{equation}
where the sum is taken over all balls $B$ in the chain.
Note that, because of the doubling property within $B_0$, the balls
$\la B$ and $B(x',\la r_B)$ have comparable measures.
The dimension condition~\eqref{eq-local-dim-cond} therefore yields
\[
|u(x')-u_{B^{0,0}}| 
   \le C_2r \sum_B \frac{\mu(B(x',\la r_B))^{1/s-1/p}}{\mu(B(x,Lr))^{1/s}}
                 \biggl( \int_{\la B} g_u^p\,d\mu \biggr)^{1/p}
=: \Sigma' + \Sigma'',
\] 
where the summations in $\Sigma'$ and $\Sigma''$ are over $B$ with 
$r_B>\rho_{i_0}$ and $r_B\le\rho_{i_0}$, respectively (and $i_0\ge0$ will
be chosen later).
Since $\la \rho_i \le \la \rho_0 = \frac{1}{2} Lr \le \frac{1}{2} \diam B_0$,
Lemma~\ref{lem-rev-doubl} implies that there exists $\theta\in(0,1)$,
independent of $x'$ and $i$, such that
\[
\mu(B(x',\la\rho_i)) \ge    \theta^{i-i_0} \mu(B(x',\la\rho_{i_0}))
\quad \text{for } i\le i_0
\]
and hence, as $1/s-1/p<0$,
\[
\Sigma' \le C_3 r \biggl( \frac{\mu(B(x',\la \rho_{i_0}))} {\mu(B(x,Lr))}
            \biggr)^{1/s-1/p}
                 \biggl( \vint_{B(x,Lr)} g_u^p\,d\mu \biggr)^{1/p}.
\]
Similarly, 
\[
\mu(B(x',\la\rho_i)) \le \theta^{i-i_0} \mu(B(x',\la\rho_{i_0}))
\quad \text{for } i>i_0
\]
and hence, as $1/s>0$,
\[
\Sigma'' \le C_4r \biggl( \frac{\mu(B(x',\la \rho_{i_0}))} {\mu(B(x,Lr))}
            \biggr)^{1/s}  M(x')^{1/p},
\]
where $M(x'):= M^*_{B(x,Lr),B_0} g_u^p (x')$ is the noncentred local
maximal function given by~\eqref{eq-def-local-max-fn}.
Here we have also used that both $x'$ and $\la\Bij$ are contained in 
$\Bhi:=B(x_{i,0},2\la\rho_i)\subset B(x,Lr)$, and 
$\tfrac52\Bhi\subset B\bigl(x,\tfrac52 Lr\bigr)\subset B_0$,  $i\ge0$.
Choosing $i_0\ge0$ so that
\[
\frac{\mu(B(x',\la \rho_{i_0}))} {\mu(B(x,Lr))}
\quad \text{is comparable to} \quad
\frac{1}{M(x')} \vint_{B(x,Lr)} g_u^p\,d\mu \le1,
\]
we can conclude that
\begin{align*}
|u(x')-u_{B^{0,0}}| \le \Sigma' + \Sigma'' 
&\le C_5r \biggl( \vint_{B(x,Lr)} g_u^p\,d\mu \biggr)^{1/s} M(x')^{1/p-1/s},
\end{align*}
which gives a lower bound for $M(x')$.
Proposition~\ref{prop-maximal-fn} then yields the level set estimate
\[
\mu(\{x'\in B(x,r): |u(x')-u_{B^{0,0}}|\ge t \}) 
   \le \frac{C_6r^q}{t^q} \mu(B(x,Lr)) 
            \biggl( \vint_{B(x,Lr)} g_u^p\,d\mu \biggr)^{q/p},
\]
which in turn implies~\eqref{eq-(q,p-PI-2la} with $B(x,Lr)$ in the
right-hand side, by \cite[Lemma~4.25]{BBbook}.

To conclude the statement of the theorem under local assumptions,
let $x_0\in X$ be arbitrary and find $r_0>0$ so that the local assumptions
hold within $B(x_0,r_0)$.
Then choose a radius $0<r_0'\le (11\la)^{-1}r_0$ so that 
$B_0':=B(x_0,r_0') \ne X$ and $\dist(x_0,X \setm B_0')=r_0'$.
For $B=B(x,r)\subset B_0'$ it then follows that $r_B\le2r_0'$ and hence 
$5\la B \subset B(x_0,r_0)$. 
The already proved first part of the theorem then
implies that~\eqref{eq-(q,p-PI-2la} holds for $B$.

Under semilocal assumptions, let $B_0':=B(x_0,r_0')\ne X$ be arbitrary and 
such that $\dist(x_0,X \setm B_0')=r_0'$.
(If $B_0'=X$, the proof is similar.)
Then $r_B\le2r_0'$ whenever $B=B(x,r)\subset B_0'$, and hence $2B\subset 5B_0'$.
Note that 
above, when proving \eqref{eq-(q,p-PI-2la} with $2 \la B$ on
the right-hand side, the 
\p-Poincar\'e inequality 
is only used to balls within $2B$ (to obtain \eqref{eq-PI-appl}), while
\eqref{eq-local-dim-cond} and the doubling property are used for balls within
$5\la B$, where $\la$ is the dilation constant
in the \p-Poincar\'e inequality
within $2B$.
Thus, to obtain~\eqref{eq-(q,p-PI-2la} for $B\subset B_0'$ 
we need to apply the \p-Poincar\'e inequality
with $\la$ and $\CPI$ determined by $5B_0'$, followed by
\eqref{eq-local-dim-cond} and the doubling property with constants
determined by $11\la B_0'$.
This can be done because of the semilocal assumptions, since the
doubling property for $\mu$ within $11\la B_0'$ implies~\eqref{eq-local-dim-cond}
within $11\la B_0'$ for some $s>0$.
\end{proof}

\begin{remark}  \label{rem-s-and-q}
There is a converse relation between $s$ and $q$ in Theorem~\ref{thm-(q,p)-PI}
as well, namely if the $(q,p)$-Poincar\'e inequality
\eqref{eq-(q,p-PI-2la} holds for all balls $B$ with $5 \la B \subset B_0$,
and $\mu$ is doubling within $B_0$,
then \eqref{eq-local-dim-cond} holds with $s=qp/(q-p)$
for all balls $X \ne B'\subset B$ with $15\la B \subset B_0$;
this follows from the proof
of \cite[Proposition~4.20]{BBbook}.
Note that the formulas $s(q)$ and $q(s)$ are inverse functions of each other,
if $p<s$.
In particular,
if we let 
\begin{align*}
\shat  &= \sup_{x \in X}\lim_{r\to0} \inf \{s>0: \eqref{eq-local-dim-cond} 
\text{ holds for all balls } B'\subset B\subset B(x,r)\}. \\
\qhat & = \sup \{q \ge p : X \text{ supports a local $(q,p)$-Poincar\'e inequality}\},
\end{align*}
then
\[
    \qhat = \begin{cases}
      \displaystyle \frac{\shat p}{\shat -p}, & \text{if } p < \shat < \infty, \\
      p, & \text{if } \shat= \infty, \\
      \infty, & \text{if } \shat \le p.
      \end{cases}
\]

Note that there need
 not exist an optimal 
 $s$ (for a given $B_0$ in Theorem~\ref{thm-(q,p)-PI}),
i.e.\ the set of values of $s$ for which \eqref{eq-local-dim-cond} holds may be an open interval,
cf.\ Example~3.1 in 
Bj\"orn--Bj\"orn--Lehrb\"ack~\cite{BBLeh1}.
Similarly there need
not be an optimal $q$.
\end{remark}

The second type of self-improvement 
for Poincar\'e inequalities is the open-ended property
due to Keith--Zhong~\cite[Theorem~1.0.1]{KZ} which
strengthens
the right-hand side of the inequality, see also 
Heinonen--Koskela--Shanmugalingam--Tyson~\cite[Theorem~12.3.9]{HKSTbook},
Eriksson-Bique~\cite{Eriksson-Bique} and
Kinnunen--Lehrb\"ack--V\"a\-h\"a\-kan\-gas--Zhong~\cite{KLVZ}.
A careful analysis of the proof
(in \cite{KZ} or \cite{HKSTbook})
shows that all the balls considered therein
lie within a constant dilation of the ball
in 
the Poincar\'e inequality under consideration.
This makes it possible to prove the following local version. 

\begin{thm} \label{thm-KZ-proper}
Assume that $p>1$ and let
$B_0=B(x_0,r_0)$ be a ball such that 
$\itoverline{B}_0$ is compact and
the \p-Poincar\'e inequality and the doubling property for $\mu$ hold
within $B_0$ in the sense of 
Definitions~\ref{def-local-doubl-mu} and~\ref{def-PI}.

Then there exist constants $C$, $\la$ and $q<p$, depending only on 
$p$, the doubling 
constant and both constants in the \p-Poincar\'e inequality within $B_0$,
such that for all balls $B$ with 
$\la B\subset B_0$,
all integrable functions $u$ on $\la B$, and all 
$q$-weak upper gradients $g$ of $u$, 
\begin{equation}    \label{eq-KZ-proper}
        \vint_{B} |u-u_B| \,d\mu
        \le C r_B \biggl( \vint_{\la B} g^{q} \,d\mu \biggr)^{1/q}.
\end{equation}
\end{thm}

In the proof below, we will use the \emph{inner metric} which is defined
by 
\[
    \din(x,y)=\inf \length(\ga),
\]
where the infimum is taken over all 
curves $\ga$ connecting $x$ and $y$.
If there are no such curves then $\din(x,y)=\infty$.
As  $X$ may be disconnected, this is not always a metric, but we will
nevertheless still use the name ``inner metric''.
Balls with respect to $\din$, defined in the obvious way,
will be denoted by $\Bin$.

\begin{proof}
Let $\la'$ be the dilation constant in the \p-Poincar\'e inequality
within $B_0$.
Lemma~\ref{lem-ex-curve-in-B0} shows that there exist $\La\ge1$ and
$L=9\La$ such that
\[
d(x,y)\le \din(x,y) \le Ld(x,y) 
\] 
whenever $B(x,2\La d(x,y))\subset B_0$.
It follows that if $2\La B(x,r)\subset B_0$, then
\begin{equation}   \label{eq-compare-balls}
B(x,r/L) \subset \Bin(x,r) \subset B(x,r) \subset \Bin(x,Lr) \subset B(x,Lr).
\end{equation}

We will now explain how the arguments in the proof of
\cite[Theorem~4.39]{BBbook} 
can be used to show that 
for every inner ball $\Bin=\Bin(x,r)$ such that 
$B(x,2r)\subset B_0$, the following inner \p-Poincar\'e inequality with
dilation constant $1$ holds:
\begin{equation}    \label{eq-PI-with-1}
\vint_{\Bin} |u-u_{\Bin}| \,d\mu
        \le C' r \biggl( \vint_{\Bin} g^p \,d\mu \biggr)^{1/p}.
\end{equation}
(Since we only assume a local \p-Poincar\'e inequality
with respect to ordinary balls, 
Theorem~4.39 in \cite{BBbook} cannot be applied directly 
and care has to be taken when comparing ordinary and inner balls.)

More precisely,  
let $\rho_0=r/2\la' L$ and $\rho_i=2^{-i} \rho_0$, $i=1,2,\ldots$\,.
For $x'\in \Bin(x,r)$, consider a $\din$-geodesic $\ga$ from $x$ 
to $x'$, i.e.\ $\ga(0)=x$ and $\ga(d')=x'$, where $d'=\din(x,x')<r$
is the length of $\ga$.
Such a geodesic exists by Ascoli's theorem and the
compactness of $\itoverline{B}_0$.
Find the smallest integer $i'\ge0$ such that $2\la'\rho_{i'}\le r-d'$.
For each $i=0,1,\ldots,i'-1$, consider all the integers $j\ge0$, such that
\[
\tij:= (1-2^{-i})d'+ j\rho_i < (1-2^{-(i+1)})d',
\]
and let $\xij=\ga(\tij)$.
There are at most $\la'L$ such $\xij$'s for each $i$.
Similarly, for $i=i'$, there are at most $2\la'L$ integers
$j\ge0$ and points $\xij=\ga(\tij)$ such that $\tij<d'$.
For $i>i'$, let $j=0$ and $\xij=x'$.

It is now easily verified that $\din(x,\xij)+\la'L\rho_i<r$
and hence, with $\Bij=B(\xij,\rho_i)$,
\[
2\la'\La \Bij\subset \la' L\Bij \subset  B(x,r) \subset B_0,
\]
so \eqref{eq-compare-balls} implies that
\[
\Bij \subset \la'\Bij \subset \Bin(\xij,\la'L\rho_i) \subset \Bin(x,r) \subset B_0.
\]
For later reference, let $y_i=\ga((1-2^{-(i+1)})d')$ when $i\le i'$ and $y_i=x'$
otherwise. 
Then each ball $\Bhij:=B(y_i,(L+1)\la'\rho_i)$ contains both $\la'\Bij$
and $x'$.
Moreover, 
\[
d(x,y_i) + \tfrac52 (L+1)\la'\rho_i < (1+2^{-(i+2)}(3+5/L))r < 2r,
\]
so $\tfrac52\Bhij \subset B(x,2r)\subset B_0$. 

Ordering the balls $\Bij$ lexicographically, we obtain a chain of balls
from $x$ to $x'$, with substantial overlaps.
Assuming that $x'$ is a Lebesgue point of $u$,
a standard telescoping argument using the \p-Poincar\'e 
inequality
for each $\Bij\subset B_0$, as in \eqref{eq-std-telescope}, 
then yields the estimate
\begin{equation}   \label{eq-telescope}
|u(x')-u_{B^{0,0}}| 
   \le C_1 \sum_B r_B \biggl( \vint_{\la'B} g_u^p\,d\mu \biggr)^{1/p},
\end{equation}
where
the sum is taken over
all balls $B$ in the above chain.
We can now estimate the measure of the set
\[
 E_t=\{x'\in \Bin(x,r): |u(x')-u_{B^{0,0}}|>t\}
\]
as follows.
Writing $t = C_2t\sum_B r_B/r$ and comparing with~\eqref{eq-telescope}, 
we can for every Lebesgue point $x'\in E_t$ find some ball $B_{x'}:=\Bij$ from the
corresponding chain so that
\begin{equation}   \label{eq-choose-Bx}
\biggl( \vint_{\la'B_{x'}} g_u^p\,d\mu \biggr)^{1/p} \ge \frac{C_3t}{r}.
\end{equation}
By the above, the corresponding ball $\Bhx:=\Bhij$ contains both $\la' B_{x'}$
and $x'$. 
Hence, using the 5-covering lemma we can from the collection $\Bhx$,
where $x'\in E_t$ are Lebesgue points of $u$, 
extract a countable pairwise disjoint subcollection
$\Bhxk$, $k=1,2,\ldots$, such that 
\[
\mu(E_t) \le \mu \biggl( \bigcup_{k=1}^\infty  5\Bhxk \biggr)
\le C_4  \sum_{k=1}^\infty \mu(\Bhxk),
\]
where in the last inequality we have used that $\tfrac52\Bhxk\subset B_0$,
so that the doubling condition can be applied.
Note also that the measures of $\la'B_{x'_k}$ and $\Bhxk$ are 
comparable.
Estimating the balls in the last sum using~\eqref{eq-choose-Bx}, 
together with the fact that the balls 
$\la' B_{x'_k}\subset \Bhxk \cap \Bin(x,r)$ are disjoint, now yields
the level set estimate
\[
t^p \mu(E_t) \le C_5r^p \sum_{k=1}^\infty \int_{\la'B_{x'_k}} g_u^p\,d\mu
\le C_6r^p \int_{\Bin(x,r)} g_u^p\,d\mu,
\]
which in turn implies~\eqref{eq-PI-with-1}, by \cite[Lemma~4.25]{BBbook}.

Next, still with respect to $\din$ and within $B_0$, the
proof in \cite[Theorem~12.3.9]{HKSTbook} (or Keith--Zhong~\cite{KZ}), which
is written for geodesic spaces, can be applied to show that
there exists $q<p$ such that for every inner ball $\Bin=\Bin(x,r)$ with 
$1280B(x,r)\subset B_0$,
\begin{equation}      \label{eq-q-PI-Bin}
\vint_{\Bin} |u-u_{\Bin}| \,d\mu
        \le C'' r \biggl( \vint_{256\Bin} g^q \,d\mu \biggr)^{1/q}.
\end{equation}
(Here it is also used that if $\Bin(x',r')\subset 256\Bin$
then $\Bin(x',r')=\Bin(x',r'')$ for some $r'' \le 512 r$, 
and hence
\[
B(x',2r'')\subset 1280 
B(x,r)\subset B_0,
\]
so \eqref{eq-PI-with-1} holds for every such inner ball 
$\Bin(x',r') \subset 256\Bin$
and can be used in the arguments leading to \cite[Theorem~12.3.9]{HKSTbook}.)
Now, by \eqref{eq-compare-balls},
\[
B(x,r/L)\subset \Bin \quad \text{and} \quad 
256\Bin\subset 256B(x,r)\subset 1280B(x,r), 
\]
all with comparable measures (depending on $L$).
Hence, \eqref{eq-q-PI-Bin}
yields~\eqref{eq-KZ-proper} with $\la=1280L$
and $B$ replaced by $B(x,r/L)$.
\end{proof}

In Heinonen--Koskela--Shanmugalingam--Tyson~\cite[Proposition~12.3.10]{HKSTbook}
it is explained how (under global assumptions) the properness of $X$
in Keith--Zhong~\cite{KZ} 
can be relaxed to local compactness,
at the price that the resulting $q$-Poincar\'e inequality
only holds for $u \in \Np(\la B)$,
which is however
enough for many applications.
A counterexample by Koskela~\cite{Koskela}  shows
that one cannot deduce a standard $q$-Poincar\'e inequality
in this case.
A similar improvement can be proved under local assumptions and
in this case we \emph{do} conclude a standard local $q$-Poincar\'e
inequality, even though $q$ may vary from ball to ball.
Under semiuniformly local assumptions there is even a 
\emph{fixed} $q<p$,
see Theorem~\ref{thm-PI-uniform-q-intro} whose
  proof is given in Section~\ref{sect-local-unif} below.

\begin{thm}  \label{thm-loc-cpt-q-PI}
If $X$ is locally compact  and supports a 
local \p-Poincar\'e inequality, $p>1$, and $\mu$ is locally doubling,
then for every $x_0\in X$ there is a ball $B_0' \ni x_0$,
 and $q<p$, such that 
a $q$-Poincar\'e inequality holds within $B_0'$
in the sense of Definition~\ref{def-PI}.

If $X$ is in addition proper and connected,
then the conclusion is semilocal,
i.e.\ it holds for all balls $B_0'\subset X$. 
\end{thm}

Note that for a semilocal conclusion it is \emph{not} enough
to assume that $X$ is locally compact and that the doubling and
Poincar\'e assumptions are semilocal.
This is another consequence of
the counterexample in Koskela~\cite{Koskela}.

\begin{proof}
Let $x_0\in X$ be arbitrary and find $r_0>0$ so that 
$\itoverline{B(x_0,r_0)}$ is compact and
the local assumptions hold within $B(x_0,r_0)$.
Let $\la$ be 
given by Theorem~\ref{thm-KZ-proper}.
Then choose a radius $0<r_0'\le (3\la)^{-1}r_0$ so that 
$B_0':=B(x_0,r_0') \ne X$ and $\dist(x_0,X \setm B_0')=r_0'$.
For $B\subset B_0'$ it then follows that $r_B\le2r_0'$ and hence 
$\la B \subset B(x_0,r_0)$.
The first statement then follows from Theorem~\ref{thm-KZ-proper}.

If $X$ is in addition proper and connected, then let $B_0':=B(x_0,r_0')$ be
arbitrary and assume that $\dist(x_0,X \setm B_0')=r_0'$ (the proof
is similar for $B_0'=X$).
Since $\la$ in Theorem~\ref{thm-KZ-proper} depends on $B_0$, we cannot
directly obtain a semilocal conclusion from it.

Instead, let $L$ be provided by 
Lemma~\ref{lem-ex-curve-in-B0} with $B_0$ replaced by $5B_0'$.  
Then for every  ball 
$B(x,r)\subset B_0'$, we have $r\le2r_0'$ and 
hence $B(x,1280Lr)\subset 2561LB_0' =:B_0$.
Because the \p-Poincar\'e inequality and the doubling property for $\mu$ hold
within $B_0$ (by Proposition~\ref{prop-semilocal-doubling-intro}
and Theorem~\ref{thm-PI-intro}), 
the 
proof of Theorem~\ref{thm-KZ-proper} 
(with constants dictated by $B_0$) 
yields \eqref{eq-q-PI-Bin}.
In particular, as $B(x,1280Lr) \subset B_0$, 
the inner $q$-Poincar\'e inequality \eqref{eq-q-PI-Bin} holds 
with  $\Bin$ replaced by $\Bin(x,Lr)$
and with constants $q<p$ and $C''>0$ depending on $B_0$.

Now, $B(x,2r) \subset 5B_0'$  (as $r\le2r_0'$)
and hence, since $X$ is proper, 
Lemma~\ref{lem-ex-curve-in-B0}
implies that
\[
B(x,r)\subset \Bin(x,Lr) 
\quad \text{and} \quad
\Bin(x,256Lr)\subset B(x,256Lr),
\] 
all with
comparable measures (by the semilocal doubling property of $\mu$
provided by Proposition~\ref{prop-semilocal-doubling-intro}).
We can therefore conclude that~\eqref{eq-KZ-proper} holds for every 
$B(x,r) \subset B_0'$ with $\la =256 L$.
\end{proof}

\section{Semiuniformly local assumptions}
\label{sect-local-unif}

A possible strengthening of
our local assumptions is to also require
uniformity in the constants and/or the radii.

\begin{deff}
The measure $\mu$ is \emph{semiuniformly locally  doubling}
if there is a (uniform) constant $C$ such that for each $x  \in X$
there is $r>0$ so that $\mu(2B)\le C \mu(B)$ 
for all balls $B \subset B(x,r)$.
If $r$ is independent of $x$, then $\mu$ is 
\emph{uniformly locally  doubling}.
\end{deff}

Note that there is no uniformity of the radii
when $\mu$ is semiuniformly locally doubling.
(Semi)uniformly local
Poincar\'e inequalities are defined similarly,
with uniform constants $C$ and $\la$.
Semiuniformly local assumptions
were used
by Holopainen--Shan\-mu\-ga\-lin\-gam~\cite{HoSh}.

It may seem more natural to impose uniformly local 
assumptions but, as we shall see,
semiuniformly local assumptions are sometimes enough.
The semiuniformly local 
assumptions 
also have the advantage that they
are inherited by open subsets, and in particular
are satisfied on 
all open subsets of spaces supporting global assumptions.
Moreover, any strictly positive continuous weight on $\R^n$
  is semiuniformly locally  doubling and supports a 
semiuniformly local  $1$-Poincar\'e
inequality.
On the other hand, our local assumptions are more general,
as seen in Example~\ref{ex-local-unif} below, and 
sufficient for many purposes, as demonstrated in this paper.
However, under semiuniformly local assumptions the constants
and exponents in the local (but not necessarily the semilocal)
results in  Section~\ref{sect-better-PI}
are also uniform. We now make this more precise. 

\begin{thm}
If $\mu$ is semiuniformly locally doubling and 
$X$ supports a local\/ \textup{(}resp.\ semiuniformly local\/\textup{)}
 \p-Poincar\'e inequality, 
 then there is $q>p$ such that
 $X$ supports a local\/ \textup{(}resp.\ semiuniformly local\/\textup{)}
 $(q,p)$-Poincar\'e inequality.
\end{thm}

\begin{proof}
This is a direct consequence of Theorem~\ref{thm-(q,p)-PI},
since the exponent $q$, given by an explicit formula,
only depends on the local doubling constant.
\end{proof}

\begin{proof}[Proof of Theorem~\ref{thm-PI-uniform-q-intro}]
This follows from Theorem~\ref{thm-loc-cpt-q-PI},
since the improvement $p-q$
in the exponent depends only on the local doubling constant and the constants
$p$, $C$ and $\la'$ in the local \p-Poincar\'e inequality within $B_0'$.
\end{proof}

Another  consequence of semiuniformly local assumptions is obvious
in Lemma~\ref{lem-ex-curve-in-B0}: the constants $\Lambda$ and $L$
therein are uniform. 
It would be interesting to know for which other results 
it is essential to require semiuniformly local
assumptions, 
especially when the consequences for the conclusions are not
just the uniformity of constants.

\begin{example} \label{ex-local-unif}
For $k \ge 3$, let
\begin{align*}
  E_k &= \bigcup_{j=2}^\infty \bigl([k-2k^{-j},k-k^{-j}]\cup[k+k^{-j},k+2k^{-j}]\bigr) 
         \times [0,k^{1-j}], \\\
  X  &= \bigl(\R \times (-\infty,0]\bigr) \cup \bigcup_{k=3}^\infty E_k,
\end{align*}
i.e.\ $X$ is the closed
lower half-plane together with countably many ``skyscraper
cities'' $E_k$ near each point $x_k=(k,0)$, $k\ge3$.
Note that each $E_k$ is self-similar with the factor $1/k$ and centre $x_k$.

Equip $X$ with the Euclidean distance and the $2$-dimensional 
Lebesgue measure~$\mu$.
Then every $x\in X$ has a neighbourhood $V_x$ (with respect to $X$) whose
interior (with respect to $\R^2$) is a uniform subdomain of $\R^2$.
This implies that $\mu$ is doubling and supports a 1-Poincar\'e inequality
on $\overline{V}_x$, 
by Theorem~4.4 in Bj\"orn--Shanmugalingam~\cite{BjShJMAA} and
Proposition~7.1 in Aikawa--Shanmugalingam~\cite{AikSh05}.
Thus $\mu$ is locally doubling and supports a local 1-Poincar\'e inequality.

For $j,k\ge4$, let $r_{j,k}=k^{-j}$ and $x_{j,k}=(k+r_{j,k},kr_{j,k})\in E_k$.
Then it is easily seen that $\mu(B(x_{j,k},kr_{j,k}))$ is comparable to
$k^{1-2j}$, while $\mu(B(x_{j,k},2kr_{j,k}))$ is comparable to $k^{2(1-j)}$.
It follows that the local doubling constant near $x_k$ is at least 
comparable to $k$.

Similarly, the ball $B(x_{j,k},3r_{j,k})$ is disconnected and 
contains the point $(k-r_{j,k},kr_{j,k})$.
Lemma~\ref{lem-rect-components} 
then implies that
the local dilation constant $\la_k$ in any Poincar\'e
inequality around the point $x_k$ satisfies $\la_k \ge \tfrac13 k$.

Letting $k\to\infty$ shows that 
$\mu$ is not semiuniformly locally doubling and does not support any 
semiuniformly local Poincar\'e inequality.
On the other hand, since $X$ is proper and connected, it follows from 
Proposition~\ref{prop-semilocal-doubling} and Theorem~\ref{thm-PI} 
that $\mu$ is semilocally doubling and supports 
a semilocal 1-Poincar\'e inequality.

If one is satisfied with a (semi)local \p-Poincar\'e inequality only for
$p>2$, rather than a 1-Poincar\'e inequality, then a simpler example can be
constructed by replacing each ``city'' $E_k$ with a wedge
\[
V_k=\{(x,y)\in\R^2: k^2|x-k|\le y \le 2k^2|x-k|\le 1 \}.
\]
In this case, the validity of a local \p-Poincar\'e inequality for $p>2$
follows from \cite[Example~A.23]{BBbook}.

It is also possible to make $X$ bounded if $\R\times(-\infty,0]$ in its 
construction is replaced by $(-1,1)\times(-1,0]$ and the ``cities'' 
or ``wedges''
are attached at the points $(1-8k^{-1},0)$ rather than at $x_k$, $k\ge8$.
In this case, $X$ is not complete and
does not support semilocal assumptions,
because they would automatically imply
global assumptions 
as $X$ is bounded,
and thus semiuniformly local assumptions would follow as well.
\end{example}

\section{Lebesgue points}
\label{sect-Leb}

Using the better Poincar\'e inequalities of
Section~\ref{sect-better-PI} it can be shown
that Newtonian functions have $L^p$-Lebesgue points
q.e.\ also under local assumptions.
By H\"older's inequality every $L^p$-Lebesgue point is 
an $L^r$-Lebesgue point for all $1 \le r \le p$.

\begin{thm} \label{thm-Leb-pt}
Assume that $p>1$, that $\mu$ is locally doubling and
that one of the following conditions is satisfied\/\textup{:}
\begin{enumerate}
\item \label{Leb-a}
$X$ is locally compact and 
supports a local \p-Poincar\'e inequality\/\textup{;}
\item \label{Leb-b}
for every $x \in X$ there are $r>0$ and $q<p$
such that 
a $q$-Poincar\'e inequality holds
within $B(x,r)$ in the sense of 
Definition~\ref{def-PI}.
\end{enumerate}
If $u \in \Nploc(X)$, then q.e.\ $x$
is an $L^p$-Lebesgue point of $u$, i.e.
\[ 
\lim_{r \to 0} \vint_{B(x,r)} |u-u(x)|^p\, d\mu=0.
\] 
\end{thm}

On metric spaces with global assumptions such results have been
obtained by Kinnunen--Latvala~\cite{KiLa02} and
Bj\"orn--Bj\"orn--Parviainen~\cite{BBP}.
Traditionally (for Sobolev functions), as well as in \cite{KiLa02} and \cite{BBP},
this is shown using the density of 
continuous functions.
Here we offer a different approach 
based on the fact that Newtonian  functions
are 
more precisely defined than arbitrary a.e.-representatives.

Note that, by Theorem~\ref{thm-locLip-dense-Om} below,
locally Lipschitz functions are dense in $\Nploc(X)$
under the assumptions in Theorem~\ref{thm-Leb-pt},
even though we do not use this fact here.
In case~\ref{Leb-a} it also follows
from Theorem~\ref{thm-when-qcont-new} below that the 
functions
in $\Nploc(X)$ are quasicontinuous. 
Even though this is not known in case~\ref{Leb-b}, 
they are still, by their Newtonian definition,
more precisely defined than arbitrary a.e.-representatives, 
which enables us to prove the existence of Lebesgue points q.e.

\begin{proof}[Proof of Theorem~\ref{thm-Leb-pt}]
Theorem~\ref{thm-loc-cpt-q-PI} shows that the assumptions~\ref{Leb-a}
imply~\ref{Leb-b}.
Thus in both cases we can 
consider a ball $B_0$ 
such that $u \in \Np(B_0)$
and a $q$-Poincar\'e inequality and the doubling property for $\mu$ hold
within $B_0$,
where $q<p$ depends on $B_0$.
Theorem~\ref{thm-(q,p)-PI} shows
that if $q<p$ is chosen close enough to $p$, then 
the Sobolev--Poincar\'e inequality 
\begin{equation}   \label{eq-(q,q*)-PI}
        \biggl(\vint_{B} |u-u_B|^{p} \,d\mu\biggr)^{1/p}
        \le C r_B \biggl( \vint_{\la B} g^{q} \,d\mu \biggr)^{1/q},
\end{equation}
with some $\la\ge2$, holds
for all balls $B$ with $\frac52\la B \subset B_0$ and 
for every upper gradient $g$ 
of $u$. 

For $x\in B_0$, let $r_j=2^{-j}$ and $v(x) = \limsup_{j\to\infty} v_j(x)$, where
\begin{equation}   \label{eq-def-v}
v_j(x) = \biggl( \vint_{B(x,r_j)} |u-u(x)|^{p}\,d\mu \biggr)^{1/p},
         \quad j=0,1,\ldots.
\end{equation}
Note that $u\in L^{p}(B_0)$ and hence,
by the Lebesgue differentiation theorem
(Theorem~\ref{thm-Leb-ae}),
$v=0$ a.e.\ in $B_0$.
We shall show that $v\in \Np(B_0)$ and hence $v=0$ q.e.\ in $B_0$, by 
\cite[Proposition~1.59]{BBbook}.

To this end, let $\ga \colon [0,l_\ga] \to B_0$ be a nonconstant
rectifiable
curve (parameterized by arc length)
and $g \in L^p(B_0)$ be an upper gradient of $u$.
By splitting $\ga$ into parts, if necessary, and by considering sufficiently
large $j$, we can assume that $\frac{1}{2}r_j \le l_\ga \le r_j$ and that
$\tfrac52\la B\subset B_0$, 
where $B:=B(x,2r_j)$ with $x=\ga(0)$ and $y=\ga(l_\ga)$
being the endpoints of~$\ga$.
Since $u\in L^{p}(B_0)$, both $v_j(x)$ and  $v_j(y)$ are finite, and hence
\begin{align*}
|v_j(x)-v_j(y)| 
& 
\le \bigl|v_j(x)-|u_B-u(x)|\bigr| + \bigl|v_j(y)-|u_B-u(y)|\bigr| 
                  + |u(x)-u(y)| \\
& 
\le \biggl( \vint_{B(x,r_j)} |u-u_B|^{p}\,d\mu \biggr)^{1/p}
        + \biggl( \vint_{B(y,r_j)} |u-u_B|^{p}\,d\mu \biggr)^{1/p}
        + \int_\ga g\,ds,
\end{align*}
where we have used that by the triangle inequality for the 
normalized $L^p$-norm,
\begin{align*}
\bigl|v_j(x)-|u_B-u(x)|\bigr| 
\kern -2em & 
\\ 
&
= \biggl| \biggl( \vint_{B(x,r_j)} |u-u(x)|^{p}\,d\mu \biggr)^{1/p}
- \biggl( \vint_{B(x,r_j)} |u_B-u(x)|^{p}\,d\mu \biggr)^{1/p} \biggr| \\
&\le \biggl( \vint_{B(x,r_j)} |(u-u(x))-(u_B-u(x))|^{p}\,d\mu \biggr)^{1/p},
\end{align*}
and similarly for $\bigl|v_j(y)-|u_B-u(y)|\bigr|$. 

The local doubling property within $B_0$ and
the Sobolev--Poincar\'e inequality~\eqref{eq-(q,q*)-PI} then
imply that 
\begin{align}
|v_j(x)-v_j(y)| &\le C' \biggl( \vint_B |u-u_B|^{p}\,d\mu \biggr)^{1/p} 
                      + \int_\ga g\,ds \nonumber \\
    & \le C'' r_j \biggl( \vint_{\la B} g^{q}\,d\mu 
                   \biggr)^{1/q} + \int_\ga g\,ds  \label{eq-q-Leb}\\
           &\le C'' r_j \inf_{z \in \la B} g_M(z)
                   + \int_\ga g\,ds  \nonumber\\
     &\le C''' \int_\ga (g_M+g) \,ds, \nonumber
\end{align}
where  $C'''$ is independent of $j$ and
$g_M^q:=M^*_{B_0,B_0}g^q$ is the noncentred local maximal function defined by
\eqref{eq-def-local-max-fn}.
Glueing together all the parts of $\ga$ and by applying 
\cite[Lemma~1.52]{BBbook} we conclude that 
$C'''(g_M+g)$ is a \p-weak upper gradient of $v$ in $B_0$.
Since $q<p$, the noncentred local
maximal operator is bounded on $L^{p/q}(B_0)$,
by Proposition~\ref{prop-maximal-fn},
which yields that
\[
\int_{B_0} g_M^p \, d\mu
\le C_0 \int_{B_0} g^p \,d\mu, 
\]
and hence $v\in \Np(B_0)$, as required.

As $X$ is Lindel\"of there is a countable cover of $X$ by balls 
$B_0$ of the type above,
and since
the capacity is countably subadditive
we conclude 
the
existence of Lebesgue points q.e.
\end{proof}

\begin{remark}   \label{rem-Lq-Leb}
  Since Lebesgue points are of a local nature, 
the proof of Theorem~\ref{thm-Leb-pt} can be modified so that
\begin{equation}     \label{eq-Leb-q(x)}
\lim_{r \to 0} \vint_{B(x,r)} |u-u(x)|^{q(x)}\, d\mu=0
\end{equation}
holds for q.e.\ $x$ and all 
\[
q(x)< \begin{cases} 
    \displaystyle \frac{s(x)p}{s(x)-p},  &\text{if } s(x)>p, \\ 
    \infty,  &\text{if } s(x)\le p,   \end{cases} 
\]
where
\[
s(x) = \lim_{r\to0} \inf \{s>0: \eqref{eq-local-dim-cond} 
\text{ holds for all balls } B'\subset B\subset B(x,r) \},
\]
cf.\ Remark~\ref{rem-s-and-q}.
If $\mu$ is semiuniformly locally doubling, then $\shat:=\sup_x s(x) < \infty$,
and we can use $\shat$ instead of $s(x)$ to find a common $q=q(x)>p$
so that 
\eqref{eq-Leb-q(x)} holds for q.e.\ $x$.

In fact, if $s(x)$ is ``attained'' then it is even possible
to reach the borderline case,
as in  Heinonen--Koskela--Shanmugalingam--Tyson~\cite[Theorem~9.2.8]{HKSTbook}.
More precisely, 
if there is $r_0>0$ such that
\[
 \eqref{eq-local-dim-cond} 
 \text{ holds for all balls } B'\subset B\subset B(x,r_0)
 \text{ with } s=s(x)>p,
\]
then we may let $q(x)=s(x)p/(s(x)-p)$.
To see this, one uses
the following estimate with $q=q(x)$ and $y \in B(x,r_0)$,
\begin{align*}
&
\limsup_{r\to0} \biggl( \vint_{B(y,r)} |u-u(y)|^{q}\,d\mu \biggr)^{1/q} 
\\
&
\quad \quad \quad
\le \limsup_{r\to0}\biggl( \vint_{B(y,r)} |u-u_{B(y,r)}|^{q}\,d\mu \biggr)^{1/q}
+ \limsup_{r\to0}|u_{B(y,r)}-u(y)| \\
& 
\quad \quad \quad
\le C \limsup_{r\to0}\biggl( r^p \vint_{B(y,r)} g_u^p\,d\mu \biggr)^{1/p}
+ \limsup_{r\to0} \vint_{B(y,r)} |u-u(y)|\,d\mu,
\end{align*}
where the $(q,p)$-Poincar\'e inequality within $B(x,r_0)$ is provided
by Theorem~\ref{thm-(q,p)-PI}.
The second term on the right-hand side
tends to $0$ q.e., as we already know that  $u$ has $L^1$-Lebesgue
points q.e., while the first term tends to $0$ q.e.\ in $B(x,r_0)$ by 
Lemma~9.2.4 in 
Heinonen--Koskela--Shanmugalingam--Tyson~\cite{HKSTbook}.

In particular, if for each $x \in X$ there is $r>0$ such that
\eqref{eq-local-dim-cond} holds for all balls 
$B'\subset B\subset B(x,r)$ with the same $s>0$ (independent of $x$),
then $u$ has
$L^q$-Lebesgue points q.e.,  with
$q=sp/(s-p)$ for $p<s$ and all $q<\infty$ for $p\ge s$,
provided that the assumptions in Theorem~\ref{thm-Leb-pt} are
fulfilled.
\end{remark}

\begin{remark}
The proof of Theorem~\ref{thm-Leb-pt} shows that the assumptions can be
further relaxed with a somewhat weaker conclusion.
Namely, if $\mu$ is locally doubling and only supports a local \p-Poincar\'e
inequality (which need not improve to a $q$-Poincar\'e inequality as
$X$ is not necessarily locally compact), then it can be
verified that the function $v$, defined by~\eqref{eq-def-v}, belongs to
$\Nq(B_0)$ for every $1\le q<p$.
To see this one replaces
$q$ by $p$ in \eqref{eq-q-Leb}, 
and then uses the $L^1$ to weak-$L^1$ boundedness \eqref{eq-max-weak-L1}
of the noncentred local
maximal operator, which yields that $g_M$ belongs to weak-$L^p(B_0)$ and thus
to $L^q(B_0)$ for all $q<p$.

An immediate consequence is
that $v=0$ outside a set of zero $q$-capacity
and hence, every $u\in\Np\loc(X)$ has Lebesgue points $q$-quasieverywhere.
Since it is not clear whether $\mu$ supports a local $q$-Poincar\'e inequality
for some $q<p$, this cannot be deduced directly 
from the inclusion $\Nploc(X) \subset \Nqloc(X)$
as we have no $q$-quasieverywhere Lebesgue point result available
for functions in  $\Nqloc(X)$ under these assumptions.
\end{remark}

\section{Density of Lipschitz functions}

Density of smooth functions is a useful property of
Sobolev spaces, with many important consequences
(which we will discuss in Section~\ref{sect-qcont}).
In metric spaces the smoothest functions to consider
are (locally) Lipschitz functions.
There are two types of density results for $\Np(X)$ 
in the literature.
The first one is due to Shanmugalingam~\cite{Sh-rev}
(it can also be found in \cite[Theorem~5.1]{BBbook}
and 
\cite[Theorem~8.2.1]{HKSTbook}).

\begin{thm} \label{thm-dense-Nages}
\textup{(Shanmugalingam~\cite[Theorem~4.1]{Sh-rev})}
If $\mu$ is globally doubling and 
supports a global \p-Poincar\'e inequality, then
Lipschitz functions are dense in $\Np(X)$.
\end{thm}

The following result was recently obtained by 
Ambrosio--Colombo--Di Marino~\cite{AmbCD}
and Ambrosio--Gigli--Savar\'e~\cite{AmbGS}.
In fact, it is not explicitly spelled out in either paper, 
but it is a direct consequence
of a combination of results in the two papers.
Below we 
explain this in some detail, see Remark~\ref{rmk-Amb}.
Note that by density we always mean norm-density in the $\Np$ norm, 
with the exception of Theorem~\ref{thm-Lip-weak-dense}.

\begin{thm}
\textup{(Ambrosio et al.~\cite{AmbCD},~\cite{AmbGS})} 
\label{thm-AmbCDGS}
Assume that $X$ 
is a complete globally doubling metric space
and that $p >1$.
Then Lipschitz functions are dense in $\Np(X)$.
\end{thm}

The main difference in 
these two results is that  the former assumes doubling 
for $\mu$ (and not just for $X$) together with
a Poincar\'e inequality, 
while  the latter requires
completeness and $p>1$. 
Note that even though the doubling property and the Poincar\'e 
inequality extend from $X$ to its closure $\itoverline{X}$, it need not
be true that $\Np(\itoverline{X})=\Np(X)$, 
cf.\ \cite[Lemma~8.2.3]{HKSTbook}.
In other words, completeness (or at least local compactness) is not
a negligible assumption.
Thus both results have their pros and cons.
Our aim in this  section is to extend 
both of these results to local assumptions and to combine them into
a unified result (when $p>1$).

\begin{remark} \label{rmk-Lip-dense}
Without both completeness and a global Poincar\'e inequality, 
Lipschitz functions are not necessarily
dense in $\Np(X)$, consider e.g.\ $X=\R\setm\{0\}$ or
the slit disc in $\R^2$, both of which support a local 1-Poincar\'e
inequality.
This also shows that the completeness assumption
in Theorem~\ref{thm-AmbCDGS} cannot be dropped nor replaced by
local compactness.

It is therefore natural to obtain
density of locally Lipschitz functions in most of our results below.
It should be mentioned that there is no known 
example of a metric space  $X$ such that 
locally Lipschitz 
functions are not dense in $\Np(X)$.
(A function $u: X \to \R$ is \emph{locally Lipschitz} 
if for every $x \in X$ there is $r>0$ such that $u|_{B(x,r)}$ is Lipschitz.)
\end{remark}

The following is a local
generalization of Theorem~\ref{thm-dense-Nages}.

\begin{thm} \label{thm-locLip-dense-Om}
If $\mu$ is locally doubling and 
supports a local \p-Poincar\'e inequality,
then locally Lipschitz functions  
are dense in $\Np\loc(X)$.
\end{thm}

Note that if $\Om$ is an open subset of $X$, then the local assumptions
are inherited by $\Om$ and hence we can also directly
conclude the density
of locally Lipschitz functions in $\Nploc(\Om$).
The same is true for Theorem~\ref{thm-locLip-dense-Om-Amb-new} below.

To prove Theorem~\ref{thm-locLip-dense-Om} we will need the following lemma.

\begin{lem}          \label{lem-Lip-dense-Np-2B}
Assume that the \p-Poincar\'e inequality and the 
doubling property for $\mu$ hold within a ball $2B_0$
in the sense of Definitions~\ref{def-local-doubl-mu} and~\ref{def-PI}.
Then every $u\in\Np(2B_0)$ can be approximated in the $\Np(B_0)$-norm 
by Lipschitz functions $u_k$ with support in $2B_0$.
Moreover, $\mu(\{x\in B_0: u(x)\ne u_k(x)\}) \to 0$ as $k\to\infty$.
\end{lem}

\begin{proof}
By Proposition~\ref{prop-doubling-mu-2/3},
$B_0$ can be covered by finitely
many balls $B_j$ with centres in $B_0$ and radii $r' =r_{B_0}/22\la$,
where $\la$ is the dilation constant in the \p-Poincar\'e inequality
within $2B_0$.
Let $\{\phi_j\}_j$ be a Lipschitz partition of unity on 
$\bigcup_j B_j$ subordinate to $2B_j$,
e.g.\ one constructed as in the proof of Theorem~\ref{thm-locLip-dense-Om}
below.

Let $u\in\Np(2B_0)$ be arbitrary and let $g\in L^p(2B_0)$ be an upper gradient
of $u$. 
Noting that $\max\{\min\{u,k\},-k\} \to u$
in $\Np(2B_0)$, as $k \to \infty$, 
we can assume without loss of generality that 
$|u| \le 1$.
Using the noncentred local maximal function in \eqref{eq-def-local-max-fn},
let 
\[
E_t = \{x\in 2B_0: M^*_{2B_0,2B_0} g^p(x)>t^p\}.
\]
Proposition~\ref{prop-maximal-fn}
then implies that 
\begin{equation}   \label{eq-tp-mu-Et-0}
t^p\mu(E_t) \le C_1 \int_{E_t} g^p\,d\mu \to0, \quad \text{as } t\to\infty.
\end{equation}
For a fixed $j$ and $x\in2B_j\setm E_t$,
we get  for all $0<\tfrac12r\le\rho\le r\le 8r'$ that
\begin{align*}
|u_{B(x,\rho)} - u_{B(x,r)}| 
&\le \vint_{B(x,\rho)} |u - u_{B(x,r)}| \,d\mu
        \le C_2 \vint_{B(x,r)} |u - u_{B(x,r)}| \,d\mu \\
     &\le C_3 r \biggl( \vint_{\lambda B(x,r)} g^p \,d\mu \biggr)^{1/p}
        \le C_3 r \bigl(M^*_{2B_0,2B_0} g^p(x)\bigr)^{1/p}
\le C_3tr,
\end{align*}
since for such radii, $\tfrac52\la B(x,r)\subset22\la B_j \subset 2B_0$.
A telescopic argument as in the proof of \cite[Theorem~5.1]{BBbook} then 
shows that the limit $\ub(x)=\lim_{r\to0} u_{B(x,r)}$ exists everywhere in
$2B_j\setm E_t$ and is $C_4t$-Lipschitz therein.
Also, by Lebesgue's differentiation theorem (Theorem~\ref{thm-Leb-ae}), 
$\ub=u$ a.e.\ in $2B_j\setm E_t$.
Using e.g.\ truncated McShane extensions, 
$\ub$ extends to a $C_4t$-Lipschitz function
$u_{t,j}$ on $2B_j$ such that $|u_{t,j}| \le 1$.
Then also $u_t=\sum_j \phi_j u_{t,j}$ equals $u$
a.e.\ in $B_0\setm E_t$.
In view of \cite[Corollary~2.21]{BBbook},
it follows that $(g+C_4t)\chi_{E_t \cap B_0}$ is a \p-weak upper
gradient of $u-u_t$ and hence
\[
\|u-u_t\|_{\Np(B_0)}^p  
   \le 2^p \mu(E_t \cap B_0) + 
     \bigl( \|g\|_{L^p(E_t \cap B_0)} + C_4 t \mu(E_t \cap B_0)^{1/p}
 \bigr)^p.
\]
Since $g\in L^p(2B_0)$, we conclude from~\eqref{eq-tp-mu-Et-0} that
$u_t\to u$ in $\Np(B_0)$. 
By construction, $u_t$ is Lipschitz in $2B_0$ and 
$\supp u_t \subset \bigcup_j 2B_j \subset 2B_0$.
\end{proof}

\begin{proof}[Proof of Theorem~\ref{thm-locLip-dense-Om}]
Let $u\in\Np\loc(X)$.
For every $x\in X$, let $B_x=B(x,r_x)$ be
a ball such that $u \in \Np(B_x)$ and such that 
the \p-Poincar\'e inequality and the 
doubling property for $\mu$ hold within $B_x$.
As $X$ 
is Lindel\"of, we can find 
a countable subcollection such that 
$X=\bigcup_{j=1}^\infty \tfrac14 B_{x_j}$.
Let $B_j=\tfrac14 B_{x_j}$, $j=1,2,\ldots$.

We construct a suitable Lipschitz partition of unity. 
For each $j$ we find $\psi_j\in\Lip(X)$ such that
$\chi_{B_j}  \le \psi_j \le \chi_{2B_j}$.
Let inductively $\phi_1=\psi_1$ and
\[
    \phi_j:= \psi_j \biggl( 1- \sum_{k=1}^{j-1} \phi_k\biggr),
\quad j\ge2.
\]
Then
$    \sum_{k=1}^j \phi_k=1$ in $B_j$,
and hence $\phi_k \equiv 0$ in $B_j$ for $k >j$, so that
\[
    \sum_{k=1}^\infty \phi_k=1 \quad \text{in } B_j.
\]
As this holds for all $B_j$ we see that $\{\phi_j\}_{j=1}^\infty$
is a partition of unity.

Since 
$4B_j = B_{x_j}$, the assumptions of 
Lemma~\ref{lem-Lip-dense-Np-2B} are satisfied for each $2B_j$ (in place 
of $B_0$).
Hence, for every $\eps>0$ and each $j$, there is $v_j\in\Lip(2B_j)$ with
\[
\|u-v_j\|_{\Np(2B_j)} \le \frac{2^{-j}\eps}{1+L_j},
\]
where $L_j$ is the  Lipschitz constant of $\phi_j$.
Then also 
\begin{align*} 
\|\phi_j (u-v_j)\|_{\Np(2B_j)}^p 
    &\le \|u-v_j\|_{L^p(2B_j)}^p 
    + (L_j\|u-v_j\|_{L^p(2B_j)} + \|g_{u-v_j}\|_{L^p(2B_j)})^p  \\
    & \le 2(1+L_j)^p \|u-v_j\|_{\Np(2B_j)}^p
    \le 2^{1-pj} \eps^p.
\end{align*}
As $v:=\sum _{j=1}^\infty \phi_j v_j$ is a locally
finite sum of Lipschitz functions, $v$ is locally Lipschitz. 
Combining this with the above estimate
we conclude that
\[
\|u-v\|_{\Np(X)} \le \sum_{j=1}^\infty \|\phi_j(u-v_j)\|_{\Np(2B_j)} 
\le 2^{1/p}\eps. 
\qedhere
\]
\end{proof}

We now turn to Theorem~\ref{thm-AmbCDGS} which we generalize in the 
following way, also taking into account Theorem~\ref{thm-dense-Nages}.
Note that the set of points, for which \ref{Amb-a} resp.\ \ref{Amb-b}
below holds,
is open.
Thus, if $X$ is connected, these two sets cannot be disjoint.

\begin{thm} \label{thm-locLip-dense-Om-Amb-new}
Let $p>1$ and 
assume that for every $x \in X$ there is 
a ball $B_x \ni x$ such that either 
\begin{enumerate}
\item \label{Amb-a}
$B_x$ is locally compact and 
globally doubling\/\textup{;} or 
\item \label{Amb-b}
the \p-Poincar\'e inequality and the 
doubling property for $\mu$ hold within $B_x$ in the
sense of Definitions~\ref{def-local-doubl-mu} and~\ref{def-PI}.
\end{enumerate}
Then locally Lipschitz functions are dense in $\Nploc(X)$.

In particular, this holds
if $p>1$ and $X$
is locally compact and locally doubling. 
\end{thm}

\begin{proof}
Without loss of generality we can replace the assumption \ref{Amb-a}
by 
\begin{enumerate}
\renewcommand{\theenumi}{\textup{(\alph{enumi}$'$)}}%
\item \label{Amb-a'}
$\itoverline{B}_x$ is compact and globally doubling.
\end{enumerate}

The proof is now the same as the proof
of Theorem~\ref{thm-locLip-dense-Om},
but with appeal to Theorem~\ref{thm-AmbCDGS} instead of
Lemma~\ref{lem-Lip-dense-Np-2B}
for the balls $B_x$ satisfying \ref{Amb-a'}.
\end{proof}

So far we have deduced the density for locally Lipschitz functions.
A natural question is when 
density 
can be obtained for Lipschitz functions.
Note 
that under the assumptions
in Theorems~\ref{thm-locLip-dense-Om}
and~\ref{thm-locLip-dense-Om-Amb-new}
it can happen that Lipschitz functions are not dense,
see Remark~\ref{rmk-Lip-dense}.

\begin{thm}   \label{thm-semi-dense}
Assume that $\mu$ is semilocally doubling and 
supports a semilocal \p-Poincar\'e inequality. 
Then Lipschitz functions with bounded 
support are dense in $\Np(X)$. 
If $X$ is also proper, then $\overline{\Lipc(X)} =\Np(X)$. 
\end{thm}

$\Lipc(X)$ denotes the space of Lipschitz
functions with compact support.

\begin{proof}
Let $u \in \Np(X)$ and $\eps >0$.
Find a sufficiently large ball $B$ with $r_B>1$ such that
$\|u\|_{\Np(X \setm \itoverline{B})} < \eps$.
By Lemma~\ref{lem-Lip-dense-Np-2B} and the semilocal assumptions,
there is $v\in\Lip(X)$ so that $\|u-v\|_{\Np(2B)}<\eps$.

For $\eta(x)=(1-\dist(x,B))_+$
we then get that
 $v_\eps:=v\eta\in\Lip(X)$ with $\supp v_\eps \subset 2B$ and
$u-v_\eps = u(1-\eta) + (u-v)\eta$.
Since
\[
g_{u(1-\eta)} \le |u| g_\eta + (1-\eta)g_u
\quad \text{and}\quad
g_{(u-v)\eta} \le \eta g_{u-v} + |u-v|g_\eta,
\]
a simple calculation yields $\|u-v_\eps\|_{\Np(X)} <6\eps$, 
which concludes the proof of the first part. 
If $X$ is also proper then $v_\eps$ has compact support and
thus $\overline{\Lipc(X)} =\Np(X)$. 
\end{proof}

\begin{proof}[Proof of Theorem~\ref{thm-Lipc-dense-intro}.]
The assumptions in 
Theorem~\ref{thm-semi-dense} are satisfied
because of 
Proposition~\ref{prop-semilocal-doubling-intro}
and Theorem~\ref{thm-PI-intro},
and hence the result follows.
\end{proof}

\begin{thm} \label{thm-Lip-dense-compl}
Assume that $X$ 
is a complete semilocally doubling metric space
and that $p>1$.
Then 
$\overline{\Lipc(X)} =\Np(X)$.
\end{thm}

Note that, by Lemma~\ref{lem-proper-equiv-complete} 
and Proposition~\ref{prop-semilocal-doubling},
a metric space is complete and semilocally doubling if and only if it is
proper and locally doubling.
A natural question is if Theorem~\ref{thm-Lip-dense-compl} holds when $X$
is only complete and  locally doubling.
The completeness in Theorem~\ref{thm-Lip-dense-compl} cannot 
be replaced by local compactness, see Remark~\ref{rmk-Lip-dense}.

\begin{proof}
This is quite 
similar to the proof of Theorem~\ref{thm-semi-dense}.
Let $u \in \Np(X)$ and $\eps >0$.
Find a sufficiently large ball $B$ with $r_B>1$ such that
$\|u\|_{\Np(X \setm \itoverline{B})} < \eps$.
By the semilocal assumptions, $\overline{2B}$ is a complete
globally doubling metric space.
Hence by Theorem~\ref{thm-AmbCDGS}
there is $v\in\Lip(\overline{2B})$ so that $\|u-v\|_{\Np(2B)}<\eps$. 
The rest of the proof is identical to the second half of the
proof of Theorem~\ref{thm-semi-dense}, since $X$ is proper
by Lemma~\ref{lem-proper-equiv-complete}.
\end{proof}

\begin{remark} \label{rmk-Amb}
We will now explain how Theorem~\ref{thm-AmbCDGS} 
follows from Ambrosio--Colombo--Di Marino~\cite{AmbCD}
and Ambrosio--Gigli--Savar\'e~\cite{AmbGS}.
In~\cite{AmbGS}, a function 
$g$ is called a \p-upper gradient of a function $f$ if
there is $\ft$ such that $\ft=f$ a.e.\
and $g$ is a \p-weak upper gradient of $\ft$ in the sense 
of Definition~\ref{deff-ug}.
In \cite{AmbGS} they also define 
\p-weak upper gradients (different from ours),
\p-relaxed upper gradients and 
\p-relaxed slopes.
Furthermore, they show that if a function $f \in L^p(X)$ has 
a gradient  $g\in L^p(X)$ in any of these four senses, then
it has one in each of the four senses and the a.e.-minimal
ones coincide.
This is shown in \cite[Theorem~7.4 and Section~8.1]{AmbGS}
assuming completeness of $X$.

In Ambrosio--Colombo--Di Marino~\cite{AmbCD}, there is a different definition
of \p-relaxed slope and the same definition of \p-weak upper gradient
as in \cite{AmbGS}. 
In~\cite[Theorem~6.1]{AmbCD} it is shown that if  $X$ is complete
and $f \in L^p(X)$  has a gradient $g \in  L^p(X)$ in either of these two senses, then
it has one in the other and the a.e.-minimal
ones coincide.
So in conclusion the a.e.-minimal gradients in $L^p(X)$ 
in all fives senses exist simultaneously
and then coincide, assuming that $f \in L^p(X)$ and that $X$ is complete.
Moreover, the Sobolev space $W^{1,p}(X)$ in \cite[Corollary~7.5]{AmbCD}
thus coincides with 
\begin{equation} \label{eq-hNp-deff}
        \hNp (X) = \{u : u=v \text{ a.e. for some } v \in \Np(X)\}
\end{equation}
(recall that functions in $\Np(X)$ are defined pointwise everywhere).

The following density result is a special case
of the equivalence result from 
\cite[Theorem~7.4 and Section~8.1]{AmbGS}.
(Note that this is not a norm-density result 
as elsewhere in this section.)

\begin{thm} \label{thm-Lip-weak-dense}
Assume that $X$ is complete and $p>1$.
Let $f \in \Np(X)$.
Then there exist Lipschitz functions $f_n$ such that 
\[
      \lim_{n \to \infty} \biggl(\int_X |f_n-f|^p \,d\mu 
        +\int_X |{\Lip f_n-g_f}|^p \,d\mu \biggr)=0,
\]
where 
\[ 
 \Lip f_n(x) := \limsup_{r\to0} \sup_{y\in B(x,r)} \frac{|f_n(y)-f_n(x)|}{r}
\] 
is the \emph{local upper Lipschitz constant} \textup{(}also called\/ 
\emph{upper pointwise dilation}\textup{)} of $f_n$,
and $g_f$ is the minimal \p-weak upper gradient of $f$
\textup{(}in the sense of Definition~\ref{deff-ug}\textup{)}.
\end{thm}

As explained at the bottom of p.\ 1 
of \cite{AmbCD},
once we know reflexivity, the  norm-density of Lipschitz functions
follows directly using Mazur's lemma.
This reflexivity was obtained in Corollary~7.5 in \cite{AmbCD}
when $X$ is complete and globally doubling,
and thus Theorem~\ref{thm-AmbCDGS} follows.

Completeness is not
assumed explicitly in \cite[Corollary~7.5]{AmbCD},
but the result relies 
on other results  
in \cite{AmbCD} (e.g.\ Theorem~7.4) 
for which completeness is assumed.
And indeed, the completeness assumption cannot be dropped or replaced
by assuming that $X$ is merely locally compact for the density result, 
since norm-density of Lipschitz functions then can fail,
see Remark~\ref{rmk-Lip-dense}.
The same counterexamples show that Theorem~\ref{thm-Lip-weak-dense}  cannot
hold (in general) in locally compact spaces 
with Lipschitz functions $f_n$.
Whether it can hold in locally compact spaces, if $f_n$ are just required 
to be locally Lipschitz, is not 
clear because
the ``partition of unity'' technique used in 
the proof of Theorem~\ref{thm-locLip-dense-Om} cannot be used
to construct such an extension,
at least not in such an easy way as here,
since it would require controlling
$\|{\Lip v - g_{f}}\|$ in terms of
$\|{\Lip (\phi_j v_j) - g_{f_j}}\|$.

A slight word of warning: one may get the impression
that as a consequence of Theorem~\ref{thm-Lip-weak-dense} one can
deduce that $g_f= \Lip f$ a.e.\ for Lipschitz functions.
This is not so, and indeed it is not true  in general, as seen by considering
e.g.\ the von Koch snowflake curve,
on which $g_f \equiv 0$ for all functions, 
because of the lack of rectifiable curves.
The equality $g_f= \Lip f$ a.e.\ for Lipschitz functions
is however true if $X$ is complete, $p>1$ and $\mu$ is 
globally doubling and supports a global \p-Poincar\'e inequality,
by Theorem~6.1 in Cheeger~\cite{Cheeg}.
\end{remark}

\begin{remark}
  Occasionally it may be interesting to know when (locally) Lipschitz 
 or continuous functions are dense in $\Np(X)$ even when $X\ne\supp\mu$,
in which case our general condition that all balls have positive measure is invalid.
  It turns out that this happens if and only if they are dense
  in $\Np(\supp \mu)$. 
  For Lipschitz and continuous functions this is Lemma~5.19\,(e) in
  \cite{BBbook}.
  For locally Lipschitz functions this can be proved similarly,
  provided that one uses the locally Lipschitz extensions 
  due to Luukkainen--V\"ais\"al\"a~\cite[Theorem~5.7]{LuukkV77}.
(Note that the class LIP in \cite{LuukkV77} consists
of locally Lipschitz functions.)

  For quasicontinuity, which will 
be discussed in the next section,
  a similar equivalence is also true, by \cite[Lemma~5.19\,(d)]{BBbook}.
\end{remark}

\section{Quasicontinuity and other consequences}
\label{sect-qcont}

Having established the density of continuous 
(or more exactly locally
Lipschitz) functions we can now draw a number of 
qualitative conclusions about Newtonian functions and capacities.

Throughout this section, $\Om \subset X$ is open.
A function $u:\Om \to \eR$ is \emph{quasicontinuous}
if for every $\eps>0$ there is an open set $G$
with $\Cp(G)<\eps$ such that $u|_{\Om \setm G}$ is continuous.
See the recent paper 
Bj\"orn--Bj\"orn--Mal\'y~\cite{BBMaly}
for several different characterizations of quasicontinuity,
and in particular that one can equivalently replace the condition
$\Cp(G) <\eps$ by $\CpOm(G) < \eps$,
where $\CpOm$ is the capacity associated with $\Np(\Om)$ rather
than with $\Np(X)$.
Note also that, in the following theorem, $X$ can be replaced by 
any open subset of $X$.
Moreover, the conditions in \ref{q-a} and \ref{q-b} below are inherited
by open subsets. 

\begin{thm} \label{thm-when-qcont-new}
Assume that $X$ is locally compact and that
one of the following conditions is satisfied\/\textup{:}
\begin{enumerate}
\item \label{q-a}
$\mu$ is locally doubling and 
supports a local \p-Poincar\'e inequality, 
\item \label{q-b}
$p>1$ and $X$ is locally doubling,
or more generally the conditions in 
Theorem~\ref{thm-locLip-dense-Om-Amb-new} are satisfied\/\textup{;}
\item \label{q-c}
continuous functions are dense in $\Np(X)$. 
\end{enumerate}

Then every $u \in \Nploc(X)$ is quasicontinuous in $X$
and hence $\Cp$ is an outer capacity, i.e.\ 
\[
\Cp(E)=\inf_{\substack{ G \supset E \\  G \text{ open}}} \Cp(G)
\quad \text{for every } E \subset X.
\]
\end{thm}

Quasicontinuity has been established for Newtonian functions
under various assumptions in 
Bj\"orn--Bj\"orn--Shanmugalingam~\cite{BBS5},
Bj\"orn--Bj\"orn--Lehrb\"ack~\cite{BBLeh1} 
and Heinonen--Koskela--Shanmugalingam--Tyson~\cite{HKSTbook}
for open subsets.
See also Shanmugalingam~\cite{Sh-rev}.
Assuming that $X$ is complete and that $\mu$ is globally doubling and supports
a global \p-Poincar\'e inequality, 
quasicontinuity was also established for functions in $\Np(U)$
when $U$ is a quasiopen subset of $X$, by 
Bj\"orn--Bj\"orn--Latvala~\cite{BBLat3} (when $p>1$)
and Bj\"orn--Bj\"orn--Mal\'y~\cite{BBMaly}.

For $p>1$, quasicontinuity also implies that
$\Cp$ is a Choquet capacity and thus, if $X$ is locally compact, that
all Borel sets are capacitable, i.e.\ 
\[
\Cp(E)=\sup_{\substack{ K \subset E \\  K \text{ compact}}} \Cp(K)
\quad \text{for every Borel set } E \subset X,
\]
see e.g.\ Aikawa--Ess\'en~\cite[Part~2, Section~10]{AE}
together with
\cite[Theorems~6.4 and 6.7\,(viii)]{BBbook}.
It should be mentioned that there is no known example 
of a Newtonian function which is not quasicontinuous,
nor of a metric space  $X$ such that continuous 
functions are not dense in $\Np(X)$.

\begin{proof}[Proof of Theorem~\ref{thm-when-qcont-new}]
By Theorems~\ref{thm-locLip-dense-Om}
and~\ref{thm-locLip-dense-Om-Amb-new}, \ref{q-a} $\imp$ \ref{q-c} and
\ref{q-b} $\imp$ \ref{q-c}. So assume that \ref{q-c} holds. 
By  \cite[Theorem~5.21]{BBbook}  
every $u \in \Nploc(X)$ has a quasicontinuous representative.
As $X$ is locally compact, Proposition~4.7 in 
Bj\"orn--Bj\"orn--Lehrb\"ack~\cite{BBLeh1} then shows that every 
$u \in \Nploc(X)$ is quasicontinuous.
The outer capacity property then follows from 
Bj\"orn--Bj\"orn--Shanmugalingam~\cite[Corollary~1.3]{BBS5} 
(or \cite[Theorem~5.31]{BBbook}).
\end{proof}

Quasicontinuity of $\Np(X)$ gets inherited to open subsets in the
  following way.
Recall the definition of $\hNp(\Om)$ in \eqref{eq-hNp-deff}.

\begin{prop}   \label{prop-qcont-consequences}
If every $u \in \Np(X)$ is quasicontinuous, then
$\Np\loc(\Om)$ 
consists exactly of those 
$u\in\hNp\loc (\Om)$ 
which are quasicontinuous, and similarly for $\Np(\Om)$.
\end{prop}

\begin{proof}
Clearly, $\Np\loc(\Om)\subset\hNp\loc(\Om)$.
Multiplying $u\in\Np\loc(\Om)$ by Lipschitz cut-off 
functions shows that for each $x \in \Om$ there is $r_x>0$ such that 
$u$ is quasicontinuous in $B(x,r_x)$.
As $X$ is Lindel\"of, 
a countable covering of $\Om$ by such balls yields quasicontinuity in
$\Om$.

Conversely, by an argument due to Kilpel\"ainen~\cite{Kilpqe},
every quasicontinuous $u\in\hNp\loc(\Om)$ is q.e.\ equal to 
a Newtonian function and hence itself in $\Np\loc(\Om)$,
cf.\ \cite[Propositions~5.22 and~5.23]{BBbook}.
\end{proof}

Quasicontinuity, or rather the outer capacity property following
from it, provides us with a short proof of the following fact,
cf.\ Kallunki--Shanmugalingam~\cite{KaSh} where it
was proved under stronger assumptions.
A similar statement is not true if 
we drop the assumption that $K$ be compact.
(Let e.g.\ $K$ be a countable dense subset of 
a ball in $\R^n$ and $p\le n$.)

\begin{prop}
Assume that $\Cp$ is an outer capacity
and that
continuous resp.\ {\rm(}locally\/{\rm)} Lipschitz functions 
are dense in $\Np(X)$. 
If $K\subset X$ is compact, then
\[
    \Cp(K) = \inf
\|u\|_{\Np(X)}^p,
\]
where the infimum is taken over all 
continuous resp.\ {\rm(}locally\/{\rm)} Lipschitz
$u$ such that $u\ge1$ on $K$.
\end{prop}

Note that if continuous
functions are dense in $\Np(X)$, then  the condition that $\Cp$ is an outer
capacity is equivalent to requiring that all functions in $\Np(X)$
are quasicontinuous, see 
Theorems~5.20, 5.31 and  Proposition~5.32 in \cite{BBbook}.

\begin{proof}
Given $\eps>0$,
there exist an open set $G\supset K$ and $u\in\Np(X)$ such that
$u=1$ on $G$ and
\begin{equation} \label{eq-CpK}
\|u\|_{\Np(X)}^p < \Cp(K)+\eps.
\end{equation}
Since $K$ is compact, there exists $0<\de\le1$ such that 
$\dist(K,X\setm G)>2\de$.
Let $\eta(x):=\min\{1,\dist(x,K)/\de\}$ 
and set $\tilde{v}=u+\eta(v-u)$,
where $v$ is continuous resp.\ (locally) Lipschitz in $X$ and
such that $\|v-u\|_{\Np(X)}<\eps\de$.
Then $\tilde{v}=1$ on $K$ and, as $g_{\eta(v-u)}\le\eta g_{v-u}+|v-u|g_\eta$,
we also have
\[
\|g_{\eta(v-u)}\|_{L^p(X)} \le \|g_{v-u}\|_{L^p(X)} + \frac{1}{\de} \|v-u\|_{L^p(X)}
    \le \frac{2}{\de} \|v-u\|_{\Np(X)} < 2\eps.
\]
It then follows that
\begin{align*}
\Cp(K) &\le \|\tilde{v}\|_{\Np(X)}^p   \\
    &\le  (\|u\|_{\Lp(X)}+ \|v-u\|_{L^p(X)}) ^p + (\|g_u\|_{\Lp(X)} 
   + \|g_{\eta(v-u)}\|_{L^p(X)}) ^p \\
    &\le  (\|u\|_{\Lp(X)}+ \eps \de) ^p + (\|g_u\|_{\Lp(X)} 
   + 2\eps) ^p,
\end{align*}
which, by \eqref{eq-CpK}, tends to $\Cp(K)$ as $\eps \to 0$.

Finally, $\tilde{v}=1-\eta+\eta v$ is continuous resp.\ (locally) Lipschitz
in $G$ (as $u\equiv1$ therein),
while $\tilde{v}=v$ is continuous resp.\ (locally) Lipschitz
in the open set 
\[
X\setm\supp(1-\eta)\supset X\setm G.
\]
It follows that $\tilde{v}$ is continuous resp.\ locally
  Lipschitz in $X$ and, as $\dist(X\setm G,\supp(1-\eta))>\de$,
  also Lipschitz in $X$ whenever $v$ is Lipschitz.
\end{proof}

\section{Conclusions for \texorpdfstring{\p}{p}-harmonic functions}

Nonlinear potential theory associated with \p-harmonic functions 
and quasiminimizers, $p>1$,
has been extensively studied during the last 20 years
on complete metric spaces equipped with globally doubling measures supporting
a global \p-Poincar\'e inequality, see e.g.\ \cite{BBbook} and the references therein.
It is therefore natural 
to see to which extent this theory can be extended
to local assumptions.

In much of this theory the properness of $X$ plays an important role and even though
some of the theory has already been developed on noncomplete spaces
(see in particular 
Kinnunen--Shanmugalingam~\cite{KiSh01}, Bj\"orn~\cite{Bj02}
and Bj\"orn--Marola~\cite{BMarola}), we will in this section
restrict ourselves
to \emph{proper $X$}.
(See Bj\"orn--Bj\"orn~\cite[Section~6]{BBnoncomp} for a similar discussion
without the properness assumption.)
\emph{We will also assume that $X$ is connected,  that $\mu$ is locally doubling
and supports a local \p-Poincar\'e inequality, and that $p>1$.}

As we have seen, it then follows from
Proposition~\ref{prop-semilocal-doubling-intro} 
and Theorem~\ref{thm-PI-intro}
that $\mu$ is semilocally doubling
and supports a semilocal \p-Poincar\'e inequality.
The results in this paper show that most of the essential tools
needed to develop the potential theory on metric spaces are available also under these
assumptions.

\begin{deff} \label{def-quasimin}
A function $u \in \Nploc(\Om)$ is a 
\emph{$Q$-quasi\/\textup{(}super\/\textup{)}minimizer}, $Q \ge 1$, in $\Om$
if 
\begin{equation} \label{eq-deff-qmin}
      \int_{\phi \ne 0} g^p_u \, d\mu 
           \le Q \int_{\phi \ne 0} g_{u+\phi}^p \, d\mu
\end{equation}
for all (nonnegative) $\phi \in \Lipc(\Om)$.

If $Q=1$ in \eqref{eq-deff-qmin} then $u$ is 
a \emph{\textup{(}super\/\textup{)}minimizer}.
A \emph{\p-harmonic function} is a continuous minimizer.
\end{deff}

See Bj\"orn~\cite[Proposition~3.2]{ABkellogg} for equivalent ways
of defining quasisuperminimizers; those equivalences also
extend to spaces with our local assumptions.
(Here Theorem~\ref{thm-semi-dense} is needed.)
Our first observation is that interior regularity is preserved
under local assumptions.
A function $u$ on $\Om$ is \emph{lsc-regularized} if
$u(x)=\essliminf_{y \to x} u(y)$ for all $x \in \Om$.

\begin{thm}   \label{thm-semiloc-int-reg}
Let $u$ be a quasi\/\textup{(}super\/\textup{)}minimizer in $\Om$.
Then $u$ has a representative $\ut$ which is continuous 
\textup{(}resp.\ lsc-regularized\/\textup{)}.

Moreover, the
weak Harnack inequalities for quasi\/\textup{(}super\/\textup{)}minimizers
hold within every ball $B_0\subset X$ 
i.e.\ for every ball $B \subset B_0$ with $\La B\subset\Om$,
where $\La$ and the weak Harnack constants depend only on $B_0$.
\end{thm}

See e.g.\  Kinnunen--Martio~\cite{KiMa02}, \cite{KiMa03},
Kinnunen--Shanmugalingam~\cite{KiSh01}, Bj\"orn--Marola~\cite{BMarola} 
and Bj\"orn--Bj\"orn~\cite{BBbook} for formulations of the
weak Harnack inequalities. 
There are various types of weak Harnack inequalities in these papers
and under different assumptions. 
In \cite{KiMa02}, \cite{KiMa03} and \cite{KiSh01}
a $q$-Poincar\'e inequality for some $q<p$ is assumed,
which under our assumptions is provided by 
Theorem~\ref{thm-loc-cpt-q-PI}.
Here $q$ will depend on the ball $B_0$.

Note that some
weak Harnack inequalities in \cite{KiMa02}, \cite{KiMa03} and \cite{KiSh01}
need to be modified, 
taking into account the dilation constant $\la$ from the \p-Poincar\'e
inequality, see Bj\"orn--Marola~\cite[Section~10]{BMarola}.
This is reflected in the constant $\La\ge1$ in 
Theorem~\ref{thm-semiloc-int-reg} in the following way:
Several of the weak Harnack inequalities in \cite{BBbook} contain 
a requirement that $50 \la B \subset \Om$.
(The factor $50$ is not the same in all the 
papers.) 
For a fixed ball $B_0\subset X$,
we let $\CPI$ and $\la$ be the constants
in the \p-Poincar\'e inequality (or $q$-Poincar\'e inequality)
within $50 B_0$, and $C_\mu$ be the doubling constant within $50 \la B_0$.
The weak Harnack inequality then holds for every ball $B \subset B_0$
provided that $50 \la B \subset\Om$ and with a constant
depending only on $\CPI$, $\la$, $C_\mu$ and $p$ (and $q$).

\begin{proof}[Proof of Theorem~\ref{thm-semiloc-int-reg}]
The arguments
in \cite[Section~4 and Theorem~5.1]{KiMa02},
\cite[Section~5]{KiMa03} and \cite{KiSh01} 
are all local, so local assumptions are enough.
They do rely on a better $q$-Poincar\'e inequality but
a suitable version is provided by Theorem~\ref{thm-KZ-proper},
as continuity is a local property.
For the lsc-regularity of quasisuperminimizers also a Lebesgue
point result is needed, which  
is justified by Theorem~\ref{thm-Leb-pt}.

For the 
weak Harnack inequalities 
it is explained above how the semilocal dependence on the constants
is achieved.
\end{proof}

Apart from that the parameters may only be semilocal, the results
in Chapters~7--14 in \cite{BBbook} all hold, with the exception of 
the Liouville theorem, which we look at in Example~\ref{ex-Liouville} below. 
This is because all the other results are of a local or semilocal nature, 
i.e.\ either
in a bounded domain or concerning a local or semilocal property.
In particular, in addition to the interior regularity in 
Theorem~\ref{thm-semiloc-int-reg}, one can prove various convergence
results, minimum and maximum principles, solve the Dirichlet and the
obstacle problem on bounded domains and obtain boundary regularity
and resolutivity for suitable boundary data.

On the other hand, for results of a global nature, such as the
Dirichlet problem
on unbounded domains (as in
Hansevi~\cite{Hansevi1}, \cite{Hansevi2}) or
global singular functions (as in
Holopainen--Shanmugalingam~\cite{HoSh}),
it is far from clear whether they hold under (semi)local assumptions.
Thus, it is precisely when studying ``global'' properties that it
is really interesting to know if the results hold with only local assumptions,
possibly (semi)uniform  ones,
or if perhaps other properties of the space play a vital role.
As an example of a ``global'' result, we will now have a look at the Liouville theorem,
and show that it does not hold in the generality considered here,
  not even under uniformly local assumptions, cf.\ Section~\ref{sect-local-unif}.

\begin{example}   \label{ex-Liouville}
Let $d\mu =w\, dx$  on $\R$, where $\alp \in \R$ and 
\[ 
   w(x)=\begin{cases}
     1, & |x| \le 1, \\
     |x|^\alp, & |x| \ge 1.
     \end{cases}
\] 
The measure $\mu$ is globally 
doubling and  
supports a uniformly local
$1$-Poincar\'e inequality, and thus
a semilocal $1$-Poincar\'e inequality. 
(And a global \p-Poincar\'e inequality 
if and only if $\alp < p-1$.)
A simple calculation shows that a function
$u$ is \p-harmonic on $(\R,\mu)$ if and only if there 
is a constant $c$ such that $u'(x)= c w(x)^{1/(1-p)}$.
If $\alp>p-1$, this gives bounded nonconstant \p-harmonic
functions on $(\R,\mu)$, namely
\[
   u(x)=\begin{cases}
     cx+b, & |x| \le 1, \\
     \displaystyle b + c \biggl( \frac{1}{\beta} +1
          - \frac{1}{\beta|x|^\beta} \biggr) \sgn x, & |x| \ge 1,
     \end{cases}
\quad \text{where } \be=\frac{\al-(p-1)}{p-1}>0
\]
and $b,c \in \R$ are arbitrary constants.
This shows that the Liouville theorem does not hold
under semilocal assumptions, nor under uniformly local ones.
\end{example}

\end{document}